\documentclass[12pt, twoside]{article}
\usepackage{times}
\usepackage{graphics}
\usepackage{theorem}
\usepackage{graphicx}
\usepackage{amsmath}
\usepackage{latexsym}
\usepackage{amssymb}
\usepackage{natbib}
\usepackage[flushmargin]{footmisc}

\theorembodyfont{\itshape}
\newtheorem{thm}{Theorem}[section]
\newtheorem{lem}{Lemma}[section]
\newtheorem{prop}{Proposition}[section]
\newtheorem{coro}{Corollary}[section]

\theorembodyfont{\rmfamily}
\newtheorem{defn}{Definition}[section]{\bf}{\rm}
\newtheorem{assumpt}{Assumption}[section]{\bf}{\rm}

\newtheorem{rem}{Remark}[section]{\itshape}{\rmfamily}
\newenvironment{proof}{\noindent{\it Proof.~~}}{\qed}

\makeatletter
\def\eqnarray{\stepcounter{equation}\let\@currentlabel=\theequation
\global\@eqnswtrue
\global\@eqcnt\z@\tabskip\@centering\let\\=\@eqncr
$$\halign to \displaywidth\bgroup\@eqnsel\hskip\@centering
  $\displaystyle\tabskip\z@{##}$&\global\@eqcnt\@ne 
  \hfil$\;{##}\;$\hfil
  &\global\@eqcnt\tw@ $\displaystyle\tabskip\z@{##}$\hfil 
   \tabskip\@centering&\llap{##}\tabskip\z@\cr}
\makeatother
\makeatletter
    \renewcommand{\theequation}{%
    \thesection.\arabic{equation}}
    \@addtoreset{equation}{section}
  \makeatother

\def\vc#1{\mbox{\boldmath $#1$}}

\newcommand{\down}[2]{\smash{\lower#2\hbox{#1}}}
\newcommand{\up}[2]{\smash{\lower-#2\hbox{#1}}}
\newcommand{\dm}{\displaystyle}

\newcommand{\qed}{\hspace*{\fill}$\Box$\rule[-10pt]{0pt}{10pt}}

\newcommand{\EE}{\mathsf{E}}
\newcommand{\PP}{\mathsf{P}}

\newcommand{\calC}{\mathcal{C}}
\newcommand{\calD}{\mathcal{D}}

\newcommand{\calL}{\mathcal{L}}
\newcommand{\calR}{\mathcal{R}}
\newcommand{\calS}{\mathcal{S}}
\newcommand{\SC}{\mathcal{SC}}

\newcommand{\bbF}{\mathbb{F}}

\newcommand{\bbM}{\mathbb{M}}
\newcommand{\bbN}{\mathbb{N}}
\newcommand{\bbR}{\mathbb{R}}

\newcommand{\bbZ}{\mathbb{Z}}

\newcommand{\rmt}{{\rm t}}
\newcommand{\rmd}{{\rm d}}
\newcommand{\rme}{{\rm e}}

\def\simhm#1{\stackrel{#1}{\sim}}

\def\ooverline#1{\overline{\overline{#1}}}

\def\simhm#1{\stackrel{#1}{\sim}}

\newcommand{\dd}[1]{\if#11 1\!\!1 
\else {\if#1C I\!\!\!C
\else {\if#1G I\!\!\!G 
\else {\if#1J J\!\!\!J 
\else {\if#1S S\!\!\!S
\else {\if#1Z Z\!\!\!Z
\else {\if#1Q O\!\!\!\!Q
\else I\!\!#1
\fi} 
\fi}
\fi}
\fi} 
\fi} 
\fi} 
\fi} 

\makeatother

\begin{document}\thispagestyle{plain} 

\hfill
{\small Last update date: \today}

{\Large{\bf
\begin{center}
A sufficient condition for the subexponential asymptotics of
GI/G/1-type Markov chains with queueing applications
%
\end{center}
}
}

\begin{center}
{
Hiroyuki Masuyama%
\footnote[2]{E-mail: masuyama@sys.i.kyoto-u.ac.jp}
}

\medskip

{\small
Department of Systems
Science, Graduate School of Informatics, Kyoto University\\
Kyoto 606-8501, Japan
}

\bigskip
\medskip

{\small
\textbf{Abstract}

\medskip

\begin{tabular}{p{0.85\textwidth}}
The main contribution of this paper is to present a new sufficient
condition for the subexponential asymptotics of the stationary
distribution of a GI/GI/1-type Markov chain without jumps from level
``infinity" to level zero. For simplicity, we call such Markov chains
{\it GI/GI/1-type Markov chains without disasters} because they are
often used to analyze semi-Markovian queues without ``disasters",
which are negative customers who remove all the customers in the
system (including themselves) on their arrivals. In this paper, we demonstrate the
application of our main result to the stationary queue length
distribution in the standard BMAP/GI/1 queue. Thus we obtain new
asymptotic formulas and prove the existing formulas under weaker
conditions than those in the literature. In addition, applying our
main result to a single-server queue with Markovian arrivals and the
$(a,b)$-bulk-service rule (i.e., MAP/${\rm GI}^{(a,b)}$/1 queue), we obatin a subexponential asymptotic
formula for the stationary queue length distribution.

\end{tabular}
}
\end{center}

\begin{center}
\begin{tabular}{p{0.90\textwidth}}
{\small

\medskip

}
\end{tabular}

\end{center}

\section{Introduction}\label{introduction}

This paper studies the subexponential asymptotics of the stationary
distribution of a GI/GI/1-type Markov chain (see, e.g.,
\citealt{He14}) without jumps from level ``infinity" to level zero.
For simplicity, we call such Markov chains {\it GI/GI/1-type Markov
  chains without disasters} because they are often used to analyze
semi-Markovian queues without ``disasters", which are negative
customers who remove all the customers in the system (including
themselves) on their arrivals. It should be noted that every
M/G/1-type Markov chain is a GI/GI/1-type Markov chain without
disasters (see, e.g., \citealt{He14}).

Several researchers have studied the subexponential asymptotics of the
stationary distributions of GI/GI/1-type Markov chains (including
M/G/1-type ones).  \citet{AsmuMoll99} derive subexponential asymptotic
formulas for the stationary distribution of a M/GI/1-type Markov chain
with subexponential level increments. \citet{Li05a} study a
GI/GI/1-type Markov chain with subexponential level increments, though
some of their asymptotic formulas are incorrect (for details, see
\citealt{Masu11}). \citet{Taki04} presents a subexponential asymptotic
formula for M/GI/1-type Markov chains, under the assumption that the
integrated tail distribution of level increments is subexponential. It
should be noted that \citet{Taki04}'s assumption does not necessarily
imply the subexponentiality of level increments themselves (see, e.g.,
Remark~3.5 in \citealt{Sigm99}). Focusing on the period of the
$G$-matrix, \citet{Masu11} establishes sufficient conditions for the
subexponential asymptotics for M/GI/1-type Markov chains, which are
weaker than those presented in the literature
(\citealt{AsmuMoll99,Li05a,Taki04}), except for being limited to the
M/G/1-type Markov chain. \citet{Masu11} also points out that
\citet{Taki04}'s derivation of the asymptotic formula implicitly
assumes the aperiodicity of the $G$-matrix. \citet{Kim12} weaken
\citet{Masu11}'s sufficient condition in the case where the $G$-matrix
is periodic. \citet{Kimu13} present a comprehensive study on the
subexponential asymptotics of GI/GI/1-type Markov chains. They study
the {\it locally} subexponential asymptotics
(\citealt{AsmuFossKors03}) as well as the (ordinarily) subexponential
asymptotics. The sufficient conditions presented in \citet{Kimu13} are
weaker than those reported in the literature mentioned above.

The main result of this paper is to present a new sufficient condition
for the subexponential asymptotics of the stationary distribution of a
GI/GI/1-type Markov chain without disasters. This sufficient condition
is weaker than the corresponding one presented in \citet{Kimu13}.

In this paper, we demonstrate the application of the main result to
the stationary queue length distribution in the (standard) BMAP/GI/1 queue (see, e.g., \citealt{Luca91}).  According to \citet{Taki00}, the stationary
queue length distribution in the BMAP/GI/1 queue is equivalent to the
stationary distribution of a certain M/G/1-type Markov
chain. Combining this fact and the main result of this paper, we
derive four subexponential asymptotic formulas for the stationary
queue length distribution. Two of the four formulas are proved under
weaker conditions than the two corresponding ones presented in
\citet{Masu09}; and the other two formulas are shown for a BMAP/GI/1
queue with consistently varying service times, which is not considered
in \citet{Masu09}. 

We also apply the main result of this paper to a single-server queue with Markovian arrivals and the
$(a,b)$-bulk-service rule, denoted by MAP/${\rm GI}^{(a,b)}$/1 queue (see, e.g., \citealt{Sing13}). For the MAP/${\rm GI}^{(a,b)}$/1 queue, we construct a GI/GI/1-type Markov chain without disasters by observing the queue length process at departure points. Thus using the main result, we obtain a subexponential asymptotic
formula for the stationary queue length distribution at departure points.
Combining the obtained formula with the relationship between the stationary queue length distribution at departure points and that at an arbitrary time point, we have a subexponential asymptotic
formula for the stationary queue length distribution at an arbitrary time point.

The rest of this paper is divided into four
sections. Section~\ref{sec-preliminary} provides basic definitions,
notation and preliminary results. Section~\ref{sec-main-results}
presents the main result of this paper. Sections~\ref{sec-app-01} and \ref{sec-app-02} discuss the applications of the main result.

\section{Preliminaries}\label{sec-preliminary}

\subsection{Basic definitions and notation}

Let $\bbZ = \{0,\pm 1,\pm 2,\dots\}$, $\bbZ_+ = \{0,1,2,\dots\}$ and
$\bbN = \{1,2,3,\dots\}$, respectively. For any distribution function
$F$ on $\bbR_+:=[0,\infty)$, let $\overline{F} = 1 - F$ and $F_{\rme}$
  denote the equilibrium distribution function of $F$, i.e.,
  $F_{\rme}(x) = \int_0^x \overline{F}(y) \rmd y / \int_0^{\infty}
  \overline{F}(y) \rmd y$ for $x \ge 0$, which is well-defined if $F$
  has a positive finite mean. For any nonnegative random variable $Y$
  with positive finite mean, let $Y_{\rme}$ denote the equilibrium
  random variable of $Y$ such that
\[
\PP(Y_{\rme} \le x) = {1 \over \EE[Y]} \int_0^x\PP(Y > y)\rmd y, 
\qquad x \in \bbZ_+;
\]
and $Y_{\rm de} = \lfloor Y_{\rme} \rfloor$, which is called the
discretized equilibrium random variable of $Y$. If $Y$ is nonnegative
integer-valued, then
\[
\PP(Y_{\rm de} = k) = {1 \over \EE[Y]} \PP(Y > k), \qquad k \in \bbZ_+.
\]

We now define $\vc{e}$ and $\vc{I}$ as the column vector of ones and
the identity matrix, respectively, with appropriate dimensions
according to the context.  The superscript ``$\rmt$" represents the
transpose operator for vectors and matrices.  The notation
$[\,\cdot\,]_{i,j}$ (rep.\ $[\,\cdot\,]_i$) denotes the $(i,j)$th
(resp.\ $i$th) element of the matrix (resp.\ vector) in the square
brackets.

For any matrix sequence $\{\vc{M}(k);k\in\bbZ\}$, let
$\overline{\vc{M}}(k) = \sum_{l=k+1}^{\infty}\vc{M}(l)$ and
$\ooverline{\vc{M}}(k) = \sum_{l=k+1}^{\infty}\overline{\vc{M}}(l)$
for $k\in\bbZ$. For any two matrix sequences
$\{\vc{M}(k);k\in\bbZ\}$ and $\{\vc{N}(k);k\in\bbZ\}$ such that
their products are well-defined, let $\{\vc{M} \ast
\vc{N}(k);k\in\bbZ\}$ denote the convolution of $\{\vc{M}(k)\}$ and
$\{\vc{N}(k)\}$, i.e.,
\[
\vc{M} \ast \vc{N}(k) 
= \sum_{l\in\bbZ} \vc{M}(k-l) \vc{N}(l)
= \sum_{l\in\bbZ} \vc{M}(l) \vc{N}(k-l),\qquad k\in\bbZ.
\]
In addition, for any square matrix sequence $\{\vc{M}(k);k\in\bbZ\}$,
let $\{\vc{M}^{\ast n}(k);k\in\bbZ\}$ ($n\in\bbN$) denote the $n$-fold
convolution of $\{\vc{M}(k)\}$ with itself, i.e.,
\[
\vc{M}^{\ast n}(k) = \sum_{l\in\bbZ}
\vc{M}^{\ast (n-1)}(k-l) \vc{M}(l),
\qquad k \in \bbZ,
\]
where $\vc{M}^{\ast0}(0) = \vc{I}$ and $\vc{M}^{\ast0}(k) = \vc{O}$
for $k \in \bbZ\setminus\{0\}$.

Finally, for simplicity, we may write $\vc{Z}(x) = o(f(x))$ and
$\vc{Z}(x) \simhm{x} \widetilde{\vc{Z}} f(x)$ to represent
\[
\lim_{x\to\infty}{\vc{Z}(x) \over f(x)} = \vc{O},
\qquad
\lim_{x\to\infty}{\vc{Z}(x) \over f(x)} = \widetilde{\vc{Z}},
\]
respectively.

The above definitions and notation for matrices are applied to vectors
and scalars in an appropriate manner.

\subsection{Stationary distribution of GI/G/1-type Markov chain}

Let $\bbM_0 = \{1,2,\dots,M_0\}$ and $\bbM = \{1,2,\dots,M\}$, where
$M_0,M\in\bbN$. We then define $\{(X_n,S_n);n\in\bbZ_+\}$ as a Markov
chain with state space $\bbF:=(\{0\} \times \bbM_0) \cup (\bbN \times
\bbM)$ and transition probability matrix $\vc{T}$, which is given by
\begin{equation}
\vc{T} = 
\left(
\begin{array}{ccccc}
\vc{B}(0)  & \vc{B}(1)  & \vc{B}(2)  & \vc{B}(3) & \cdots
\\
\vc{B}(-1) & \vc{A}(0)  & \vc{A}(1)  & \vc{A}(2) & \cdots
\\
\vc{B}(-2) & \vc{A}(-1) & \vc{A}(0)  & \vc{A}(1) & \cdots
\\
\vc{B}(-3) & \vc{A}(-2) & \vc{A}(-1) & \vc{A}(0) & \cdots
\\
\vdots     & \vdots     &  \vdots    & \vdots    & \ddots
\end{array}
\right),
\label{defn-T}
\end{equation}
where $\vc{B}(0)$ and $\vc{A}(0)$ in the diagonal blocks are $M_0
\times M_0$ and $M \times M$ matrices, respectively. Each element of
$\vc{T}$ is specified by two nonnegative integers $(k,i) \in \bbF$,
where the first variable $k$ is called {\it level} and the second one
$i$ is called {\it phase}.

Throughout this paper, we make the following assumption:
\begin{assumpt}\label{assu-1}
(i) $\vc{T}$ is irreducible and stochastic; (ii) $\sum_{k=1}^{\infty} k
  \vc{B}(k) \vc{e} < \infty$; (iii) $\vc{A}:=\sum_{k\in\bbZ} \vc{A}(k)$
  is irreducible and stochastic; (iv) $\sum_{k\in\bbZ} |k| \vc{A}(k) <
  \infty$; (v) $\sigma := \vc{\pi} \sum_{k\in\bbZ} k \vc{A}(k) \vc{e}
  < 0$, where $\vc{\pi}:=(\pi_i)_{i\in\bbM}$ is the stationary
  probability vector of $\vc{A} := \sum_{k\in\bbZ} \vc{A}(k)$.
\end{assumpt}

\begin{rem}\label{rem-assu-1}
$\vc{T}$ is positive recurrent if and only if $\sigma < 0$ and
  $\sum_{k=1}^{\infty} k\vc{B}(k)\vc{e} < \infty$, provided that
  $\vc{T}$ and $\vc{A}$ are irreducible and stochastic (see, e.g.,
  \citealt[Chapter~XI, Proposition~3.1]{Asmu03}). Therefore
  Assumption~\ref{assu-1} is equivalent to condition~(I) of
  Assumption~2 in \citet{Kimu13}.
\end{rem}

\begin{rem}
For $k \in \bbN$, we have $\vc{B}(-k)\vc{e} +
\sum_{l=-k+1}^{\infty}\vc{A}(l)\vc{e}= \vc{e}$.  Thus
condition (iii) of Assumption~\ref{assu-1} implies
$\lim_{k\to\infty}\vc{B}(-k) = \vc{O}$, which shows that the one-step
transition probability from level ``infinity" to level zero is equal
to zero, i.e., no ``disasters" happen in the context of queueing models.
\end{rem}

Let $\vc{x} := (\vc{x}(0),\vc{x}(1),\vc{x}(2),\dots)$ denote the
unique stationary probability vector of $\vc{T}$, where $\vc{x}(0)$
(resp.\ $\vc{x}(k)$; $k\in\bbN$) is a $1 \times M_0$ (resp.\ $1 \times
M$) subvector of $\vc{x}$ corresponding to level zero (resp.\ level
$k$).  To characterize $\vc{x} =
(\vc{x}(0),\vc{x}(1),\vc{x}(2),\dots)$, we introduce $R$-matrices. Let
$\vc{R}_0(k)$ and $\vc{R}(k)$ ($k \in \bbN$) denote $M_0 \times M$ and
$M \times M$ matrices, respectively, such that
\[
[\vc{R}_0(k)]_{i,j}
= \EE\left[ \sum_{n=1}^{T_{< k}} \dd{1}(X_n = k, S_n = j) 
\mid X_0 = 0, S_0 = i \right],
\]
and for any fixed $\nu\in\bbN$,
\[
\phantom{}[\vc{R}(k)]_{i,j}
= \EE\left[ \sum_{n=1}^{T_{< k+\nu}} \dd{1}(X_n = k+\nu, S_n = j) 
\mid X_0 = \nu, S_0 = i \right],
\]
where $T_{<k} = \inf\{n \in \bbN; X_n < k \le X_m~(m=1,2,\dots,n-1)\}$
and $\dd{1}(\cdot)$ denotes the indicator function of the event in the
parentheses. For convenience, let $\vc{R}_0(0) = \vc{O}$ and
$\vc{R}(0) = \vc{O}$. It then follows (see, e.g.,
\citealt{Kimu13,Li05a}) that
\[
\vc{x}(k) = \vc{x}(0) \vc{R}_0 * \vc{F}(k),\qquad k \in \bbN,
\]
where
\begin{equation}
\vc{F}(k) = \sum_{n=0}^{\infty}\vc{R}^{\ast n}(k),
\qquad k\in\bbZ_+.
\label{def-F(k)}
\end{equation}
Thus we have
\begin{equation}
\overline{\vc{x}}(k) =
\vc{x}(0) \overline{ \vc{R}_0 * \vc{F}}(k), \qquad k\in\bbZ_+,
\label{eq-xk-R0-Gamma}
\end{equation}
and especially,
\begin{equation}
\overline{\vc{x}}(0) = 
\vc{x}(0) \vc{R}_0 (\vc{I} - \vc{R} )^{-1},
\label{eq-bar-x_0-R}
\end{equation}
where $\vc{R} = \sum_{k=1}^{\infty}\vc{R}(k)$ and $\vc{R}_0 =
\sum_{k=1}^{\infty}\vc{R}_0(k)$.

For the discussion in the next section, we need some more definitions
and preliminary results. Let $\vc{G}(k)$ ($k \in \bbN$) denote an $M
\times M$ matrix such that for any fixed $\nu \in \bbN$,
\[
[\vc{G}(k)]_{i,j}
= \PP(X_{T_{<k+\nu}} = \nu, S_{T_{<k+\nu}} = j \mid X_0 = k+\nu, S_0 = i),
\qquad k \in \bbN.
\]
Let $\vc{\Phi}(0)$ denote an $M \times M$ matrix such that for any
fixed $\nu \in \bbN$,
\[
[\vc{\Phi}(0)]_{i,j}
= \PP(S_{T_{\downarrow \nu}} = j \mid X_0 = \nu, S_0 = i),
\]
where $T_{\downarrow \nu} = \inf\{n \in \bbN; X_n = \nu <
X_m~(m=1,2,\dots,n-1)\}$. Note here that
$\sum_{n=0}^{\infty}(\vc{\Phi}(0))^n=(\vc{I} - \vc{\Phi}(0))^{-1}$
exists because $\vc{T}$ is irreducible.  Since Assumption~\ref{assu-1}
is equivalent to condition~(I) of Assumption~2 in \citet{Kimu13} (see
Remark~\ref{rem-assu-1}), we have the following result:
\begin{prop}[\citealt{Kimu13}, Lemma~3.1.1]\label{prop-sigma}
Under Assumption~\ref{assu-1}, 
\[
\sigma = -
\vc{\pi}
(\vc{I} - \vc{R} ) (\vc{I} - \vc{\Phi}(0)) 
\sum_{k=1}^{\infty}k\vc{G}(k) \vc{e} \in (-\infty,0).
\]
\end{prop}

Let $\vc{L}(k)$ ($k \in \bbN$) denote an $M \times M$ matrix such that
for any fixed $\nu \in \bbN$,
\[
[\vc{L}(k)]_{i,j}
=\PP(S_{T_{\downarrow \nu}} = j \mid X_0 = k+\nu, S_0 = i),
\qquad k \in \bbN.
\]
We then have
\[
\vc{L}(k) 
= \sum_{m=1}^k \vc{G}^{\ast m}(k),
\qquad k \in \bbN.
\]
In terms of $\vc{L}(k)$, the matrices
$\vc{R}_0(k)$ and $\vc{R}(k)$ are expressed as
\begin{align}
&&
\vc{R}_0(k)
&=
\left[
\vc{B}(k) + \sum_{m=1}^{\infty} \vc{B}(k+m) \vc{L}(m) \right]
(\vc{I} - \vc{\Phi}(0))^{-1},
& k &\in \bbN,
&&
\nonumber
\\
&&
\vc{R}(k) 
&=
\left[
\vc{A}(k)  + \sum_{m=1}^{\infty} \vc{A}(k+m) \vc{L}(m)\right]
(\vc{I} - \vc{\Phi}(0))^{-1}, 
& k &\in \bbN.
&&
\label{eqn-R-A}
\end{align}

The following proposition is used to prove Lemma~\ref{lem-AD-BD} in
the next section.
\begin{prop}[\citealt{Kimu13}, Lemma~3.1.2]\label{prop-lim-L(ntau+l)}
If Assumption~\ref{assu-1} holds, then
\[
\lim_{n\to \infty} \sum_{l=0}^{\tau-1}\vc{L}(n\tau+l)
= \tau \vc{e} \vc{\psi},
\]
where 
\begin{equation}
\vc{\psi} = \vc{\pi}(\vc{I} - \vc{R} ) (\vc{I} - \vc{\Phi}(0))/(-\sigma),
\label{defn-psi}
\end{equation}
and $\tau$ denotes the period of an Markov
additive process with kernel $\{\vc{A}(k);k\in\bbZ\}$ (see Appendix~B
in \citealt{Kimu10}).
\end{prop}

\begin{rem}
Proposition~\ref{prop-sigma} implies that $\vc{\psi}$ is finite.
\end{rem}

\subsection{Long-tailed distributions}\label{appen-heavy-tailed}

We begin with the definitions of the long-tailed class and
higher-order long-tailed classes.
\begin{defn}\label{defn-long-tailed}
A nonnegative random variable $U$ and its distribution $F_U$ are said
to be long-tailed if $\PP(U>x) > 0$ for all $x \ge 0$ and $\PP(U>x+y)
\simhm{x} \PP(U>x)$ for some (thus all) $y > 0$. The class of
long-tailed distributions is denoted by $\calL$.
\end{defn}

\begin{defn}\label{defn-higher-order-L}
A nonnegative random variable $U$ and its distribution $F_U$ are said
to be the $\mu$ th-order long-tailed if $U^{1/\mu} \in \calL$, where
$\mu \ge 1$. The class of the $\mu$th-order long-tailed distributions
is denoted by $\calL^{\mu}$. Further if $U \in \calL^{\mu}$
(resp.\ $F_U \in \calL^{\mu}$) for all $\mu \ge 1$, we write $U \in
\calL^{\infty}$ (resp.\ $F_U \in \calL^{\infty}$) and call $U$
(resp.\ $F_U$) infinite-order long-tailed.
\end{defn}

The basic properties of the higher-order long-tailed classes
(including the long-tailed class) are summarized in
Proposition~\ref{rem-higher-order-L} below.
\begin{prop}[\citealt{Masu13}, Lemmas~A.1--A.3]\label{rem-higher-order-L}
\hfill
\begin{enumerate}
\item $\calL^{\mu_2} \subset \calL^{\mu_1}$ for $1 \le \mu_1 < \mu_2$.
\item If $U \in \calL^{\mu}$ ($\mu\ge1$), then $\PP(U > x)
= \exp\{-o(x^{1/\mu})\}$.
\item $U \in \calL^{\mu}$ ($\mu\ge1$) if and only if $\PP(U > x - \xi
  x^{1-1/\mu}) \simhm{x} \PP(U > x)$ for some (thus all) $\xi \in \bbR
  \backslash \{0\}$.
\end{enumerate}
\end{prop}

Next we introduce the subexponential class, which is the largest
tractable subclass of $\calL$.
\begin{defn}[\citealt{Gold98,Sigm99}]\label{defn-subexp}
A nonnegative random variable $U$ and its distribution $F_U$ are said
to be subexponential if $\PP(U>x) > 0$ for all $x \ge 0$ and
\[
\PP(U_1 + U_2 > x) \simhm{x} 2 \PP(U>x),
\]
where $U_i$'s ($i=1,2,\dots$) are independent copies of $U$.  The
class of subexponential distributions is denoted by $\calS$.
\end{defn}

\begin{rem}
The class $\calS$ includes Pareto, heavy-tailed Weibull, lognormal,
Burr, and loggamma distributions, etc (see, e.g., \citealt{Gold98}).
\end{rem}

The following proposition is used several times in the subsequent
sections.
\begin{prop}[\citealt{Masu11}, Proposition~A.3]\label{Masu11-prop}
Let $\{\vc{M}(k);k\in\bbZ_+\}$ and $\{\vc{N}(k);k\in\bbZ_+\}$ denote
finite-dimensional nonnegative matrix sequences such that their
convolution $\{\vc{M}\ast\vc{N}(k);k\in\bbZ_+\}$ is well-defined and
$\vc{M}:= \sum_{k=0}^{\infty}\vc{M}(k)$ and $\vc{N}:=
\sum_{k=0}^{\infty}\vc{N}(k)$ are finite. Suppose that for some random
variable $U \in \calS$,
\[
\lim_{k\to\infty}{\overline{\vc{M}}(k) \over \PP(U > k)} 
= \widetilde{\vc{M}} \ge \vc{O},
\qquad
\lim_{k\to\infty}{\overline{\vc{N}}(k) \over \PP(U > k)} 
= \widetilde{\vc{N}} \ge \vc{O},
\]
where $\widetilde{\vc{M}} = \widetilde{\vc{N}} = \vc{O}$ is allowed.
We then have
\[
\lim_{k\to\infty}{ \overline{\vc{M} \ast \vc{N}}(k) \over \PP(U > k) } 
= \widetilde{\vc{M}} \vc{N} + \vc{M} \widetilde{\vc{N}}.
\]
\end{prop}

Finally we describe two subclasses of $\calS$, which are used to apply
the main result of this paper to the BMAP/GI/1 queue in
Section~\ref{sec-app-01}.
\begin{defn}[\citealt{Shne06}]\label{defn-SC'}
A nonnegative random variable $U$ and its distribution function $F_U$
and cumulative hazard function $Q_U:=-\log \overline{F}_U$ belong to
the subexponential concave class $\SC$ with index $\alpha$ ($0 <
\alpha < 1$) if the following hold: (i) $Q_U$ is eventually concave;
(ii) $\log x = o(Q_U(x))$; and (iii) there exist some $x_0 > 0$ such
that $Q_U(x)/x^{\alpha}$ is nonincreasing for all $x \ge x_0$, i.e.,
\[
{Q_U(x) \over Q_U(u)}  \le \left( {x \over u} \right)^{\alpha},
\qquad x \ge u \ge x_0.
\]
The subexponential concave class with index $\alpha$ is denoted by
$\SC_{\alpha}$.
\end{defn}

\begin{rem}
$\SC_{\alpha} \subset \calL^{1/\beta}$ for all $0 < \alpha < \beta \le
  1$ (see Lemma A.6 in \citealt{Masu13}). In addition, typical examples
  of $Q_U \in \SC$ are (i) $Q_U(x) = (\log x)^{\gamma}x^{\alpha}$ and
  (ii) $Q_U(x) = (\log x)^{\beta}$, where $0 < \alpha < 1$, $\beta >
  1$ and $\gamma \in \bbR$. See Appendix~A.2 in \citet{Masu13} for
  further remarks.
\end{rem}

\begin{defn}\label{defn-class-C}
 A nonnegative random variable $U$ and its distribution function $F_U$
 belong to the consistent variation class $\calC$ if
 $\overline{F}_U(x) > 0$ for all $x \ge 0$ and
\[
\lim_{v\downarrow1}\liminf_{x\to\infty}
{\overline{F}_U(vx) \over \overline{F}_U(x)}
= 1
~~
\mbox{or equivalently,}~~
\lim_{v\uparrow1}\limsup_{x\to\infty}
{\overline{F}_U(vx) \over \overline{F}_U(x)}
= 1.
\]
\end{defn}

\begin{rem}\label{rem-class-C}
It is known that (i) $\calC \subset \calL^{\infty}$ (see Lemma A.4 in
\citealt{Masu13}); (ii) $\calR \subset \calC \subset \calL \cap \calD
\subset \calS$ where $\calD$ and $\calR$ denote the dominated
variation class and the regular variation class, respectively (see,
e.g., the introduction of \citealt{Ales08}).
\end{rem}

\section{Main Result}\label{sec-main-results}

Before presenting the main result, we first show a related result.
\begin{prop}[\citealt{Kimu13}, Theorem~3.1.1]\label{prop-thm3.1-Kimu13}
Suppose that (i) Assumption~\ref{assu-1} is satisfied; and (ii) there
exists some random variable $U$ in $\bbZ_+$ with positive finite mean
such that $U_{\rm de} \in \calS$ and
\begin{equation}
\lim_{k \to \infty} {\overline{\vc{A}}(k)\vc{e} \over \PP(U > k)}
= { \vc{c}_A \over \EE[U] },
\qquad
\lim_{k \to \infty} {\overline{\vc{B}}(k)\vc{e} \over \PP(U > k)}
= { \vc{c}_B \over \EE[U] },
\label{eqn-lim-A(k)e-Kimu13}
\end{equation}
where $\vc{c}_A$ and $\vc{c}_B$ are $M \times 1$ and $M_0 \times 1$
nonnegative vectors, respectively, satisfying $\vc{c}_A \neq \vc{0}$
or $\vc{c}_B \neq \vc{0}$. We then have
\[
\lim_{k\to \infty} {\overline{\vc{x}}(k) \over \PP(U_{\rm de} > k) }
= {\vc{x}(0)\vc{c}_B + \overline{\vc{x}}(0)\vc{c}_A
 \over -\sigma } \cdot \vc{\pi}.
\]
\end{prop}

In this section, we present a more general result than the above
proposition.  For this purpose, we make the following assumption:
\begin{assumpt}\label{assu-tail-A(k)e-B(k)e}
There exists some random variable $Y$ in $\bbZ_+$ such that
\begin{equation}
\lim_{k \to \infty} {\ooverline{\vc{A}}(k)\vc{e} \over \PP(Y > k)}
= \vc{c}_A,
\qquad
\lim_{k \to \infty} {\ooverline{\vc{B}}(k)\vc{e} \over \PP(Y > k)}
= \vc{c}_B,
\label{eqn-lim-A(k)e}
\end{equation}
where $\vc{c}_A$ and $\vc{c}_B$ are $M \times 1$ and $M_0 \times 1$
nonnegative vectors, respectively, satisfying $\vc{c}_A \neq \vc{0}$
or $\vc{c}_B \neq \vc{0}$.
\end{assumpt}

\begin{rem}\label{rem-assumpt}
We suppose that (\ref{eqn-lim-A(k)e-Kimu13}) holds for some some
random variable $U$ in $\bbZ_+$ with positive finite mean ($U_{\rm de}
\in \calS$ is not necessarily assumed). It then follows from
(\ref{eqn-lim-A(k)e-Kimu13}) that
\[
\lim_{k \to \infty} {\overline{\vc{A}}(k)\vc{e} \over \PP(U_{\rm de} = k)}
= \vc{c}_A,
\qquad
\lim_{k \to \infty} {\overline{\vc{B}}(k)\vc{e} \over\PP(U_{\rm de} = k)}
= \vc{c}_B,
\]
which yield
\[
\lim_{k \to \infty} {\ooverline{\vc{A}}(k)\vc{e} \over \PP(U_{\rm de} > k)}
= \vc{c}_A,
\qquad
\lim_{k \to \infty} {\ooverline{\vc{B}}(k)\vc{e} \over\PP(U_{\rm de} > k)}
= \vc{c}_B.
\]
Thus Assumption~\ref{assu-tail-A(k)e-B(k)e} holds for $Y = U_{\rm de}$.
\end{rem}

The following theorem is the main result of this paper.
\begin{thm}\label{thm-original}
Suppose that (i) Assumption~\ref{assu-1} is satisfied; and (ii)
Assumption~\ref{assu-tail-A(k)e-B(k)e} holds for some $Y \in
\calS$. We then have
\begin{equation}
\lim_{k\to \infty} {\overline{\vc{x}}(k) \over \PP(Y > k) }
= {\vc{x}(0)\vc{c}_B + \overline{\vc{x}}(0)\vc{c}_A
 \over -\sigma } \cdot \vc{\pi}.
\label{eq-bar-x-Y}
\end{equation}
\end{thm}

Before proving Theorem~\ref{thm-original}, we compare the above
theorem with Proposition~\ref{prop-thm3.1-Kimu13}.  According to
Remark~\ref{rem-assumpt}, condition (ii) of
Proposition~\ref{prop-thm3.1-Kimu13} is sufficient for condition (ii)
of Theorem~\ref{thm-original}.  On the other hand, the latter do not
imply the former. To confirm this, we suppose that
(\ref{eqn-lim-A(k)e}) holds for a random $Y$ in $\bbZ_+$ such that
\begin{equation}
\PP(Y > k)
= 
\left\{
\begin{array}{ll}
\PP(U_{\rm de} > 2n),   						
& k = 2n,~n \in \bbZ_+,
\\
\dm{1 \over 2}\left\{\PP(U_{\rm de} > 2n) + \PP(U_{\rm de} > 2n+1)\right\}, 
& k = 2n+1,~n \in \bbZ_+,
\end{array}
\right.
\label{defn-Y}
\end{equation}
where $U$ is a random variable in $\bbZ_+$ such that $U \in \calS$ and
$U_{\rm de} \in \calS$ (see \citealt{Gold98} and also Definition~A.3 and
Proposition A.2 in \citealt{Masu11}). It follows from $U_{\rm de} \in
\calS$ and (\ref{defn-Y}) that $\PP(Y > k) \simhm{k} \PP(U_{\rm de} >
k)$ and thus $Y \in \calS$ (\citealt[Proposition 2.8]{Sigm99}),
which shows that condition (ii) of Theorem~\ref{thm-original} holds
for $Y \in \calS$ defined in (\ref{defn-Y}).

Note here that (\ref{eqn-lim-A(k)e}), (\ref{defn-Y}) and $U \in \calS
\subset \calL$ yield
\begin{eqnarray}
\overline{\vc{A}}(2n)\vc{e}
&=& \ooverline{\vc{A}}(2n-1)\vc{e} - \ooverline{\vc{A}}(2n)\vc{e}
\nonumber
\\
&\simhm{n}& 
{\vc{c}_A \over 2}
\left\{ \PP(U_{\rm de} > 2n-2) + \PP(U_{\rm de} > 2n-1) \right\}
- \vc{c}_A \PP(U_{\rm de} > 2n)
\nonumber
\\
&=& \vc{c}_A
\bigg[
{1 \over 2}
\left\{ \PP(U_{\rm de} > 2n-2) - \PP(U_{\rm de} > 2n-1) \right\}
\nonumber
\\
&& {} \qquad 
+ \PP(U_{\rm de} > 2n-1) - \PP(U_{\rm de} > 2n)
\bigg]
\nonumber
\\
&=& \vc{c}_A
\bigg[
{1 \over 2} \PP(U_{\rm de} = 2n-1) + \PP(U_{\rm de} = 2n)
\bigg]
\nonumber
\\
&=& 
\vc{c}_A
\left(
{1 \over 2} {\PP(U > 2n-1) \over \EE[U]} + {\PP(U > 2n) \over \EE[U]}
\right)
\simhm{n}
{3 \vc{c}_A \over 2} {\PP(U > 2n)  \over \EE[U] },
\label{lim-overline{A}(2n)e}
\end{eqnarray}
and 
\begin{eqnarray}
\overline{\vc{A}}(2n+1)\vc{e}
&=& \ooverline{\vc{A}}(2n)\vc{e} - \ooverline{\vc{A}}(2n+1)\vc{e}
\nonumber
\\
&\simhm{n}& \vc{c}_A\PP(U_{\rm de} > 2n) - 
{\vc{c}_A \over 2}\left\{\PP(U_{\rm de} > 2n) + \PP(U_{\rm de} > 2n+1)\right\}
\nonumber
\\
&=&
{\vc{c}_A \over 2}
\left\{ \PP(U_{\rm de} > 2n) - \PP(U_{\rm de} > 2n+1) \right\}
\nonumber
\\
&=& {\vc{c}_A \over 2} \PP(U_{\rm de} = 2n+1)
=
{\vc{c}_A \over 2} { \PP(U > 2n+1) \over \EE[U] }.
\label{lim-overline{A}(2n+1)e}
\end{eqnarray} 
The equations (\ref{lim-overline{A}(2n)e}) and
(\ref{lim-overline{A}(2n+1)e}) show that
$\lim_{k\to\infty}\overline{\vc{A}}(k)\vc{e}/\PP(U > k)$ does not
exist and thus condition (ii) of Proposition~\ref{prop-thm3.1-Kimu13}
does not hold. Consequently, Theorem~\ref{thm-original} is a more
general result than Proposition~\ref{prop-thm3.1-Kimu13}.

In what follows, we prove Theorem~\ref{thm-original}. To this end, we
establish three lemmas.
\begin{lem}\label{lem-AD-BD}
Suppose that Assumption~\ref{assu-1} is satisfied. If
Assumption~\ref{assu-tail-A(k)e-B(k)e} holds for some $Y \in \calL$,
then
\begin{eqnarray}
\lim_{k \to \infty} \sum_{m = 1}^{\infty} 
{\overline{\vc{A}}(k+m) \vc{L}(m)
 \over \PP(Y > k) }
&=& 
{\vc{c}_A\vc{\pi}(\vc{I} - \vc{R} ) (\vc{I} - \vc{\Phi}(0)) 
\over -\sigma },
\label{limit-AD}
\\
\lim_{k \to \infty} \sum_{m = 1}^{\infty} 
{\overline{\vc{B}}(k+m) \vc{L}(m)
 \over \PP(Y > k) } 
&=& 
{\vc{c}_B\vc{\pi}(\vc{I} - \vc{R} ) (\vc{I} - \vc{\Phi}(0)) 
\over -\sigma }.
\label{limit-BD}
\end{eqnarray}
\end{lem}

\proof See Appendix~\ref{proof-lem-AD=BD}. \qed

\begin{lem}\label{lem-R-R0}
Suppose that Assumption~\ref{assu-1} is satisfied. If
Assumption~\ref{assu-tail-A(k)e-B(k)e} holds for some $Y \in \calL$,
then
\begin{eqnarray}
\lim_{k \to \infty} {\overline{\vc{R}}(k) \over \PP(Y > k) }
&=&
{\vc{c}_A\vc{\pi} ( \vc{I} - \vc{R} )  \over -\sigma
},
\label{eq-R-A-Y}
\\
\lim_{k \to \infty} {\overline{\vc{R}_0}(k) \over \PP(Y > k) }
&=&
{\vc{c}_B\vc{\pi} ( \vc{I} - \vc{R} ) 
\over -\sigma
}.
\label{eq-R0-B-Y}
\end{eqnarray}
\end{lem}
\proof
From (\ref{eqn-R-A}), we have
\begin{equation}
\overline{\vc{R}}(k) 
= \left[ 
\overline{\vc{A}}(k) + \sum_{m = 1}^{\infty} 
\overline{\vc{A}}(k+m) \vc{L}(m) \right](\vc{I} - \vc{\Phi}(0))^{-1}.
\label{eqn-overline{R}(k)}
\end{equation}
Further it follows from (\ref{eqn-lim-A(k)e}) and $Y \in \calL$ that
\[
\lim_{k \to \infty} {\overline{\vc{A}}(k) \over \PP(Y > k) }
\le \lim_{k \to \infty} 
{\ooverline{\vc{A}}(k-1)\vc{e}\vc{e}^{\rmt} - \ooverline{\vc{A}}(k)\vc{e}\vc{e}^{\rmt} 
\over \PP(Y > k) } = \vc{O}.
\]
Thus (\ref{eqn-overline{R}(k)}) yields
\begin{equation}
\lim_{k \to \infty} {\overline{\vc{R}}(k) \over \PP(Y > k) }
=
\lim_{k \to \infty} \sum_{m = 1}^{\infty} 
{\overline{\vc{A}}(k+m) \vc{L}(m)
 \over \PP(Y > k) } ( \vc{I} - \vc{\Phi}(0) )^{-1} .
\label{eq-R-A-add1}
\end{equation}
Substituting (\ref{limit-AD}) into (\ref{eq-R-A-add1}), we obtain
(\ref{eq-R-A-Y}).  Similarly, we can prove (\ref{eq-R0-B-Y}). \qed

\begin{lem}\label{lem-gamma} 
Suppose that Assumption~\ref{assu-1} is satisfied. If
Assumption~\ref{assu-tail-A(k)e-B(k)e} holds for some $Y \in \calS$,
then
\begin{equation}
\lim_{k \to \infty} {\overline{\vc{F}}(k) \over \PP(Y > k)}
= 
 {(\vc{I} - \vc{R} )^{-1}
\vc{c}_A\vc{\pi} \over -\sigma
}.
\label{eq-gamma-Y}
\end{equation}
\end{lem}
\proof It follows from (\ref{def-F(k)}) that
\begin{equation}
\sum_{k=0}^{\infty}\vc{F}(k) = (\vc{I} - \vc{R})^{-1}.
\label{eqn-sum-F(k)}
\end{equation}
Further combining (\ref{def-F(k)}) with Lemma~6 in \citet{Jele98} and
(\ref{eqn-sum-F(k)}) yields
\begin{eqnarray*}
\lim_{k \to \infty} {\overline{\vc{F}}(k) \over \PP(Y > k)}
&=&
 (\vc{I} - \vc{R} )^{-1} 
 \lim_{k \to \infty} {\overline{\vc{R}}(k) 
\over \PP(Y > k) }(\vc{I} - \vc{R})^{-1}.
\end{eqnarray*}
From this and (\ref{eq-R-A-Y}), we have (\ref{eq-gamma-Y}). \qed

\medskip

We now provide the proof of Theorem~\ref{thm-original}.

\noindent
{\it Proof of Theorem~\ref{thm-original}.\ }
Applying Proposition~\ref{Masu11-prop} to
(\ref{eq-xk-R0-Gamma}) and using (\ref{eq-R0-B-Y}),
(\ref{eq-gamma-Y}) and (\ref{eqn-sum-F(k)}), we obtain
\[
\lim_{k \to \infty} {\overline{\vc{x}}(k) \over \PP(Y > k) } 
=
{\vc{x}(0) \over -\sigma}
\left[ \vc{c}_B \vc{\pi}  +
\vc{R}_0(\vc{I} - \vc{R})^{-1} \vc{c}_A \vc{\pi}
\right].
\]
Substituting (\ref{eq-bar-x_0-R}) into the above equation yields
(\ref{eq-bar-x-Y}). \qed

\section{Application to BMAP/GI/1 Queue}\label{sec-app-01}

This section discusses the application of the main result to the
standard BMAP/G/1 queue.

\subsection{Model description}\label{subsec-BMAP}

We first introduce the batch Markovian arrival process (BMAP)
(\citealt{Luca91}). Let $\{J(t);t\ge0\}$ denote a Markov chain with
state space $\bbM=\{1,2,\dots,M\}$, which is called background Markov
chain. Let $\{N(t);t \ge 0\}$ denote the counting process of arrivals
from the BMAP.  We assume that the bivariate process
$\{(N(t),J(t));t\ge0\}$ is a Markov chain with state space $\bbZ_+
\times \bbM$ and the following infinitesimal generator $\vc{Q}$:
\begin{equation}
\vc{Q}
= \left(
\begin{array}{ccccc}
\vc{C} & \vc{D}(1) & \vc{D}(2) & \vc{D}(3) & \cdots
\\
\vc{O} & \vc{C}   & \vc{D}(1) & \vc{D}(2) & \cdots
\\
\vc{O} & \vc{O}   & \vc{C}   & \vc{D}(1) & \cdots
\\
\vc{O} & \vc{O}   & \vc{O}   & \vc{C}   & \ddots
\\
\vdots & \vdots   &\vdots    & \ddots   & \ddots
\end{array}
\right),
\label{defn-Q}
\end{equation}
where $\vc{D}(k) \ge \vc{O}$ ($k \in \bbN$), $[\vc{C}]_{i,i} < 0$ ($i
\in \bbM$), $[\vc{C}]_{i,j} \ge 0$ ($i \neq j$, $i,j \in \bbM$) and
$\left( \vc{C} + \sum_{k=1}^{\infty}\vc{D}(k) \right)\vc{e} = \vc{0}$.
Thus the BMAP is characterized by the rate matrices
$\{\vc{C},\vc{D}(1),\vc{D}(2),\dots\}$.

Let $\widehat{\vc{D}}(z) = \sum_{k=1}^{\infty}z^k\vc{D}(k)$ and
$\vc{D} = \widehat{\vc{D}}(1) = \sum_{k=1}^{\infty}\vc{D}(k)$. It then
follows from (\ref{defn-Q}) that
\[
\EE[z^{N(t)} \dd{1}(J(t) = j) \mid J(0) = i] 
= \left[ \exp\{(\vc{C} + \widehat{\vc{D}}(z))t \} \right]_{i,j},
\quad i,j \in \bbM,~t \ge 0,
\]
and that $\vc{C} + \vc{D}$ is the infinitesimal generator of the
background Markov chain $\{J(t);t\ge0\}$.  For analytical convenience,
we assume that $\vc{C} + \vc{D}$ is irreducible, and then define
$\vc{\varpi}:=(\varpi_i)_{i\in\bbM} > \vc{0}$ as the unique stationary
probability vector of $\vc{C} + \vc{D}$. In this setting, the mean
arrival rate, denoted by $\lambda$, is given by
\begin{equation}
\lambda = \vc{\varpi} \sum_{k=1}^{\infty}\vc{D}(k) \vc{e},
\label{defn-lambda}
\end{equation}
which is assumed to be strictly positive (i.e., $\lambda > 0$) in
order to exclude a trivial case.

Customers are served on the first-come-first-served basis, and their
service times are independent and identically distributed
(i.i.d.)\ according to distribution function $H$ with mean $h \in
(0,\infty)$ and $H(0) = 0$. We assume that the offered load $ \rho :=
\lambda h > 0$ satisfies
\[
\rho < 1,
\]
which ensures that the BMAP/GI/1 queue is stable
(\citealt{Loyn62}).

Let $\vc{y}(k)$ denote a $1 \times M$ vector
such that $[\vc{y}(k)]_i = \PP(L = k, J=i)$ for $(k,i) \in \bbZ_+ \times
\bbM$, where $L$ and $J$ denote generic random variables for the
number of customers in the system and the state of the background
Markov chain, respectively, in steady state. It is known that
$\vc{y}:=(\vc{y}(0),\vc{y}(1),\vc{y}(2),\dots)$ is the stationary
probability vector of the following transition probability matrix of
M/G/1 type (\citealt{Taki00}):
\begin{equation}
\vc{T}_{\rm M/G/1}
:=
\left(
\begin{array}{ccccc}
\vc{P}(0) & \vc{P}(1) & \vc{P}(2) & \vc{P}(3) & \cdots
\\
\vc{P}(0) & \vc{P}(1) & \vc{P}(2) & \vc{P}(3) & \cdots
\\
\vc{O}    & \vc{P}(0) & \vc{P}(1) & \vc{P}(2) & \cdots
\\
\vc{O}    & \vc{O}    & \vc{P}(0) & \vc{P}(1) & \cdots
\\
\vdots    & \vdots    &  \vdots   & \vdots   & \ddots
\end{array}
\right),
\label{defn-T_M/G/1}
\end{equation}
where $\vc{P}(k)$ ($k \in \bbZ_+$) denotes an $M \times M$ matrix such
that
\begin{equation}
\widehat{\vc{P}}(z)
:= \sum_{k=0}^{\infty}z^k \vc{P}(k)
= \int_0^{\infty} \exp\{(\vc{C} + \widehat{\vc{D}}(z))x\} \rmd H(x).
\label{defn-hat{P}(z)}
\end{equation}
It is
easy to see that $\vc{T}_{\rm M/G/1}$ is equivalent to $\vc{T}$ in
(\ref{defn-T}) with
\begin{eqnarray}
\vc{A}(k) = 
\left\{
\begin{array}{ll}
\vc{P}(k+1), & k \ge -1,
\\
\vc{O}, & k \le -2,
\end{array}
\right.
\qquad
\vc{B}(k) = 
\left\{
\begin{array}{ll}
\vc{P}(k), & k \in \bbZ_+,
\\
\vc{P}(0),   & k = -1,
\\
\vc{O}, & k \le -2.
\end{array}
\right.
\label{relation-AB-P}
\end{eqnarray}
Note here that (\ref{defn-lambda}), (\ref{defn-hat{P}(z)}) and $\rho =
\lambda h$ yield
\begin{equation}
\vc{\varpi}\sum_{k=1}^{\infty} k \vc{P}(k) \vc{e}
= \vc{\varpi} \widehat{\vc{P}}{}'(1) \vc{e} 
= \vc{\varpi}\sum_{k=1}^{\infty} k \vc{D}(k) \vc{e} 
\cdot \int_0^{\infty} x \rmd H(x) = \lambda h = \rho.
\label{eqn-rho}
\end{equation}

We now define $\vc{P}_{\rme}(k)$ ($k \in \bbZ_+$) as an $M \times M$
matrix such that
\begin{equation}
\widehat{\vc{P}}_{\rme}(z)
:= \sum_{k=0}^{\infty}z^k \vc{P}_{\rme}(k)
= \int_0^{\infty} \exp\{(\vc{C} + \widehat{\vc{D}}(z))x\} \rmd H_{\rme}(x),
\label{defn-hat{A}_e(z)}
\end{equation}
where $H_{\rme}$ is the equilibrium distribution of the service
time distribution $H$. We then have the following
lemma:
\begin{lem}\label{lem-P_e}
\begin{equation}
\overline{\vc{P}}(k)\vc{e} 
= h\cdot\vc{P}_{\rme} \ast \overline{\vc{D}}(k)\vc{e},
\qquad k \in \bbZ_+.
\label{eqn-ooverline{A}(k)e}
\end{equation}
\end{lem}

\proof Post-multiplying both sides of (\ref{defn-hat{A}_e(z)}) by
$-\vc{C} - \widehat{\vc{D}}(z)$ and integrating the right hand side by
parts yield
\begin{equation}
\widehat{\vc{P}}_{\rme}(z)(-\vc{C} - \widehat{\vc{D}}(z))
= h^{-1}(\vc{I} - \widehat{\vc{P}}(z)),
\qquad |z| < 1.
\label{eqn-hat{A}_e(z)}
\end{equation}
It follows from (\ref{eqn-hat{A}_e(z)}) and $-\vc{C}\vc{e} =
\vc{D}\vc{e} = \widehat{\vc{D}}(1)\vc{e}$ that
\begin{equation}
\widehat{\vc{P}}_{\rme}(z)
{\widehat{\vc{D}}(1)\vc{e} -  \widehat{\vc{D}}(z)\vc{e} \over 1 - z}
= h^{-1}{\vc{e} -  \widehat{\vc{P}}(z)\vc{e} \over 1 - z},
\qquad |z| < 1.
\label{eqn-hat{A}_e(z)-02}
\end{equation}
Note here that
\[
\sum_{k=0}^{\infty}z^k\overline{\vc{D}}(k)\vc{e}
= {\widehat{\vc{D}}(1)\vc{e} -  \widehat{\vc{D}}(z)\vc{e} \over 1 - z},
\qquad
\sum_{k=0}^{\infty}z^k\overline{\vc{P}}(k)\vc{e}
= {\vc{e} -  \widehat{\vc{P}}(z)\vc{e} \over 1 - z}.
\]
Substituting these equations into (\ref{eqn-hat{A}_e(z)-02}), we have
\[
\widehat{\vc{P}}_{\rme}(z)\sum_{k=0}^{\infty}z^k\overline{\vc{D}}(k)\vc{e}
= h^{-1}\sum_{k=0}^{\infty}z^k\overline{\vc{P}}(k)\vc{e},
\]
and thus
\[
\overline{\vc{P}}(k)\vc{e} 
= h\cdot \sum_{l=0}^k\vc{P}_{\rme}(l) \overline{\vc{D}}(k-l)\vc{e},
\qquad k \in \bbZ_+,
\]
which shows that (\ref{eqn-ooverline{A}(k)e}) holds.
\qed

\subsection{Asymptotic formulas for the queue length}

In this subsection, we present some subexponential asymptotic formulas
for the stationary queue length distribution of the BMAP/GI/1 queue.
For this purpose, we use the following result:
\begin{coro}\label{coro-BMAP}
Suppose that there exists some random variable $Y$ in $\bbZ_+$ such
that $Y \in \calS$ and
\begin{equation}
\lim_{k\to\infty}{\ooverline{\vc{P}}(k)\vc{e} \over \PP(Y > k)} 
= \vc{c} \ge \vc{0}, \neq \vc{0}.
\label{asymp-P(k)}
\end{equation}
We then have
\begin{equation}
\overline{\vc{y}}(k) 
\simhm{k} {\vc{\varpi}\vc{c} \over 1 - \rho}\vc{\varpi} 
\cdot \PP(Y > k).
\label{asymp-y(k)}
\end{equation}

\end{coro}

\proof Recall that $\vc{T}_{\rm M/G/1}$ in (\ref{defn-T_M/G/1}) is
equivalent to $\vc{T}$ in (\ref{defn-T}) with block matrices
$\vc{A}(k)$ and $\vc{B}(k)$ ($k \in \bbZ$) satisfying
(\ref{relation-AB-P}). Recall also that $\vc{\varpi}$ is the
stationary probability vector of $\vc{C} + \vc{D}$. Thus
(\ref{defn-hat{P}(z)}) implies that $\vc{\varpi}$ satisfies
$\vc{\varpi}\widehat{\vc{P}}(1) = \vc{\varpi}$ and corresponds to the
stationary probability vector $\vc{\pi}$ of
$\vc{A}=\sum_{k\in\bbZ}\vc{A}(k)$. Combining these facts with
(\ref{relation-AB-P}), (\ref{eqn-rho}) and (\ref{asymp-P(k)}), we have
\begin{eqnarray*}
\ooverline{\vc{A}}(k)\vc{e}
&\simhm{k}& \ooverline{\vc{B}}(k)\vc{e}
\simhm{k} \vc{c} \cdot \PP(Y > k),
\\
\sigma 
&=& \vc{\varpi}\sum_{k=0}^{\infty}(k-1)\vc{P}(k) \vc{e}
= \rho - 1.
\end{eqnarray*}
Therefore (\ref{asymp-y(k)}) follows from Theorem~\ref{thm-original}
and $[\vc{y}(0)]_i + [\overline{\vc{y}}(0)]_i = \PP(J=i) = \varpi_i$
($i\in\bbM$). \qed

\medskip

In the following, we
consider three cases: (i) the service time distribution is
light-tailed; (ii) second-order long-tailed; and (iii)
consistently varying.

\subsubsection{Light-tailed service time}

Let $G$ denote a random variable in $\bbZ_+$ such that
$\PP(G = 0) = 0$ and
\begin{equation}
\PP(G = k) = {\vc{\varpi}\vc{D}(k)\vc{e} \over \lambda_G},
\qquad k \in \bbN,
\label{defn-G}
\end{equation}
where $\lambda_G$ is the arrival rate of batches, i.e., $\lambda_G =
\vc{\varpi}\vc{D}\vc{e}$. From the definition of $G$, we have $\EE[G]
= \lambda / \lambda_G$ and thus
\begin{equation}
\PP(G_{\rm de} > k)
= {\vc{\varpi}\ooverline{\vc{D}}(k)\vc{e} \over \lambda},
\qquad k \in \bbZ_+.
\label{eqn-Y_e}
\end{equation}

We now make the following assumption:
\begin{assumpt}\label{assumpt-G}
There exists some $\widetilde{\vc{d}}_G \ge \vc{0}, \neq \vc{0}$ such
that
\begin{equation}
\lim_{k\to\infty}{\ooverline{\vc{D}}(k)\vc{e} \over \PP(G_{\rm de} > k)}
= \widetilde{\vc{d}}_G.
\label{lim-ooverline{D}(k)e}
\end{equation}
\end{assumpt}

\begin{thm}\label{thm-BMAP-GI-1-00}
Suppose that $H$ is light-tailed, i.e., $\int_0^{\infty} \rme^{\delta
  x} \rmd H(x) < \infty$ for some $\delta > 0$. Further if
Assumption~\ref{assumpt-G} holds and $G_{\rm de} \in \calS$, then
\begin{equation}
\ooverline{\vc{P}}(k)\vc{e} 
\simhm{k} h \widehat{\vc{P}}_{\rme}(1) \widetilde{\vc{d}}_G\cdot \PP(G_{\rm de} > k),
\label{asymp-ooverline{A}(k)-03}
\end{equation}
and
\begin{equation}
\PP(L > k, J = i)
\simhm{k} {\rho \over 1 - \rho}\varpi_i \cdot \PP(G_{\rm de} > k).
\label{asymp-overline{p}(k)-00}
\end{equation}
\end{thm}

\proof It follows from (\ref{defn-hat{A}_e(z)}) and
$\vc{\varpi}(\vc{C} + \vc{D}) = \vc{0}$ that
\begin{equation}
\vc{\varpi} \widehat{\vc{P}}_{\rme}(1) = \vc{\varpi},
\label{add-eqn-03a}
\end{equation}
and from (\ref{eqn-Y_e}) and (\ref{lim-ooverline{D}(k)e}) that
\begin{equation}
\vc{\varpi}\widetilde{\vc{d}}_G = \lambda.
\label{add-eqn-03b}
\end{equation}
Thus if (\ref{asymp-ooverline{A}(k)-03}) holds, then
(\ref{add-eqn-03a}), (\ref{add-eqn-03b}) and Corollary~\ref{coro-BMAP}
yield
\[
\overline{\vc{y}}(k)
\simhm{k} {\rho \over 1 - \rho} \vc{\varpi} 
\cdot\PP(G_{\rm de} > k),
\]
which shows that (\ref{asymp-overline{p}(k)-00}) holds.

In what follows, we prove (\ref{asymp-ooverline{A}(k)-03}). Let
$\vc{\varLambda}(k)$ ($k \in \bbZ_+$) denote
\begin{equation}
\vc{\varLambda}(k) = 
\left\{
\begin{array}{ll}
\vc{I} + \theta^{-1}\vc{C},	& k = 0,
\\
\theta^{-1}\vc{D}(k), 		& k \in \bbN,
\end{array}
\right.
\label{defn-varGamma(k)}
\end{equation}
where $\theta = \max_{j \in \bbM}|[\vc{C}]_{j,j}|$. We then rewrite
(\ref{defn-hat{A}_e(z)}) as
\begin{eqnarray*}
\sum_{k=0}^{\infty} z^k \vc{P}_\rme(k)
= 
\int_0^{\infty} \sum_{n=0}^{\infty} 
\rme^{-\theta x}{(\theta x)^n \over n!}\rmd H_\rme(x)
\left[\sum_{k=0}^{\infty} z^k \vc{\varLambda}(k)\right]^n,
\end{eqnarray*}
which implies that
\begin{equation}
\overline{\vc{P}}_\rme(k)
= 
\int_0^{\infty} \sum_{n=1}^{\infty} 
e^{-\theta x}{(\theta x)^n \over n!}\rmd H_\rme(x)
\overline{\vc{\varLambda}^{\ast n}}(k),
\qquad k \in \bbZ_+.
\label{eqn-overline{A_k}}
\end{equation}

According to Corollary 3.3 in \cite{Sigm99}, $G_{\rm de} \in \calS \subset \calL$
implies $\PP(G > k) = o(\PP(G_{\rm de} > k))$. It thus follows from
(\ref{defn-G}), (\ref{eqn-Y_e}), (\ref{defn-varGamma(k)}) and
$\vc{\varpi} > \vc{0}$ that for $i\in\bbM$,
\begin{eqnarray}
[\overline{\vc{\varLambda}}(k)\vc{e}]_i 
&=& {\lambda_G \over \theta} {[\overline{\vc{D}}(k)\vc{e}]_i  \over \lambda_G }
\le {\lambda_G \over \theta\varpi_i} 
{\vc{\varpi}\overline{\vc{D}}(k)\vc{e}  \over \lambda_G }
\nonumber
\\
&=& {\lambda_G \over \theta\varpi_i}\PP(G > k)
= o(\PP(G_{\rm de} > k)).
\end{eqnarray}
Using this and Proposition~\ref{Masu11-prop}, we obtain 
\begin{equation}
\overline{\vc{\varLambda}^{\ast n}}(k) = o(\PP(G_{\rm de} > k)),
\qquad n \in \bbN.
\label{add-eqn-01}
\end{equation}
Note here that $H$ is light-tailed if and only if $H_\rme$ is
light-tailed.  Therefore similarly to the proof of Lemma 3.5 in
\cite{Masu09}, we can readily prove from (\ref{eqn-overline{A_k}}) and
(\ref{add-eqn-01}) that
\begin{equation}
\overline{\vc{P}}_\rme(k)= o(\PP(G_{\rm de} > k)).
\label{add-eqn-02}
\end{equation}
As a result, we obtain (\ref{asymp-ooverline{A}(k)-03}) by applying
Proposition~\ref{Masu11-prop} to (\ref{eqn-ooverline{A}(k)e}) and
using (\ref{lim-ooverline{D}(k)e}) and (\ref{add-eqn-02}). \qed

\medskip

\citet{Masu09} present a similar result:
\begin{prop}[\citealt{Masu09}, Theorem~3.2]\label{Masu09-prop-00}
Suppose that (i) $H$ is light-tailed; and (ii) there exists some
$\widetilde{\vc{D}} \ge \vc{O}, \neq \vc{O}$ such that
$\overline{\vc{D}}(k) \simhm{k} \widetilde{\vc{D}} \PP(G > k)$.
Further if $G \in \calS$ and $G_{\rm de} \in \calS$, then
(\ref{asymp-overline{p}(k)-00}) holds.
\end{prop}

Theorem~\ref{thm-BMAP-GI-1-00} shows that the condition $G \in \calS$
in Proposition~\ref{Masu09-prop-00} is not necessary for the
subexponential asymptotic formula (\ref{asymp-overline{p}(k)-00}). In
addition, condition (ii) of Proposition~\ref{Masu09-prop-00} implies
Assumption~\ref{assumpt-G} whereas its converse does not. This fact is
confirmed similarly to the comparison of Theorem~\ref{thm-original}
and Proposition~\ref{prop-thm3.1-Kimu13} in
Section~\ref{sec-main-results}. As a result, the conditions of
Proposition~\ref{Masu09-prop-00} are more restrictive than those of
Theorem~\ref{thm-BMAP-GI-1-00}.

\subsubsection{Second-order long-tailed service time}

\begin{thm}\label{thm-BMAP-GI-1-01}
Suppose that (i) $H_{\rme} \in \calL^{\mu}$ for some $\mu \ge 2$; and
(ii) $\sum_{k=1}^{\infty}\rme^{Q(k)}\vc{D}(k) < \infty$ for some
cumulative hazard function $Q \in \SC$ such that $x^{1/\mu} =
O(Q(x))$. We then have
\begin{equation}
\overline{\vc{P}}_{\rme}(k) 
\simhm{k}  \vc{e} \vc{\varpi} \cdot \overline{H}_{\rme}(k/\lambda).
\label{asymp-overline{A}_{e}(k)}
\end{equation}
In addition, if (iii) $H_{\rme} \in \calS$, then
\begin{equation}
\ooverline{\vc{P}}(k)\vc{e} 
\simhm{k} 
\rho \vc{e} \cdot 
\overline{H}_{\rme}(k/\lambda),
\label{asymp-ooverline{A}(k)}
\end{equation}
and 
\begin{equation}
\PP(L > k, J = i)
\simhm{k} {\rho \over 1 - \rho}\varpi_i \cdot \overline{H}_{\rme}(k/\lambda).
\label{asymp-overline{p}(k)-01}
\end{equation}
\end{thm}

\begin{rem}\label{rem-thm-BMAP-GI-1-01}
Condition (i) implies that $\overline{H}_\rme(x) =
\exp\{-o(x^{1/\mu})\}$ (see
Proposition~\ref{rem-higher-order-L}~(ii)). Further condition (ii)
implies that $\overline{\vc{D}}(k) = o(\exp\{-\delta k^{1/\mu}\})$ for
some $\delta > 0$. Thus $\overline{\vc{D}}(k) =
o(\overline{H}_\rme(k))$.
\end{rem}

\noindent
{\it Proof of Theorem~\ref{thm-BMAP-GI-1-01}.}~ Let $T$ denote a
nonnegative random variable distributed with $H_{\rme}$ independently
of BMAP $\{\vc{C},\vc{D}(1),\vc{D}(2),\dots\}$. We can readily obtain
\begin{equation}
\PP(N(T) > k \mid J(0) = i) \simhm{k} \PP(T > k/\lambda ),
\qquad i \in \bbM,
\label{asymp-N(T)-01}
\end{equation}
by following the proof of Lemma~3.1 in \citet{Masu09} and using
Corollary~\ref{appendix-coro-01} instead of Lemma~2.1 in
\citet{Masu09}. Further similarly to the proof of Lemma~3.2 in
\citet{Masu09}, we can prove from (\ref{asymp-N(T)-01}) that
\[
\PP(N(T) > k, J(T) = j \mid J(0) = i) \simhm{k} \varpi_j \PP(T > k/\lambda ),
\qquad i,j \in \bbM,
\]
which shows that (\ref{asymp-overline{A}_{e}(k)}) holds.

Next we prove (\ref{asymp-ooverline{A}(k)}). According to
Remark~\ref{rem-thm-BMAP-GI-1-01}, $\overline{\vc{D}}(k) =
o(\exp\{-\delta k^{1/\mu}\})$ for some $\delta > 0$, which implies
that
\begin{eqnarray*}
\ooverline{\vc{D}}(k) 
&\le& o(\exp\{-(\delta/2) k^{1/\mu}\})
\sum_{l=k+1}^{\infty} \exp\{-(\delta/2) l^{1/\mu} \}
\nonumber
\\
&=& o(\exp\{-(\delta/2) k^{1/\mu}\}).
\end{eqnarray*}
Thus
since $\overline{H}_{\rme}(k/\lambda) = \exp\{-o(k^{1/\mu})\}$ (see
Remark~\ref{rem-thm-BMAP-GI-1-01}), we obtain
\begin{equation}
\ooverline{\vc{D}}(k) = o(\overline{H}_{\rme}(k/\lambda)).
\label{eqn-ooverline{D}(k)-01}
\end{equation}
Applying Proposition~\ref{Masu11-prop} to (\ref{eqn-ooverline{A}(k)e})
and using (\ref{asymp-overline{A}_{e}(k)}) and
(\ref{eqn-ooverline{D}(k)-01}) yield
\begin{eqnarray*}
\ooverline{\vc{P}}(k)\vc{e} 
&\simhm{k}& h
\vc{e} \vc{\varpi} \sum_{k=0}^{\infty}\overline{\vc{D}}(k)\vc{e} 
\cdot \overline{H}_{\rme}(k/\lambda)
= \rho \vc{e} \cdot \overline{H}_{\rme}(k/\lambda),
\end{eqnarray*}
where the last equality is due to (\ref{defn-lambda}) and $\rho =
\lambda h$. Therefore we have (\ref{asymp-ooverline{A}(k)}).

Finally, from (\ref{asymp-ooverline{A}(k)}) and
Corollary~\ref{coro-BMAP}, we have
\[
\overline{\vc{y}}(k)
\simhm{k} {\rho \over 1 - \rho} \vc{\varpi} \cdot \overline{H}_{\rme}(k/\lambda),
\]
which shows that (\ref{asymp-overline{p}(k)-01}) holds.
\qed

\medskip

We now compare Theorem~\ref{thm-BMAP-GI-1-01} with a similar result
presented in \citet{Masu09}, which is as follows:
\begin{prop}[\citealt{Masu09}, Theorem~3.1]\label{Masu09-prop-01}
If (i) $H \in \calL^2$ and $H_{\rme} \in \calS$; and (ii)
$\sum_{k=1}^{\infty}\rme^{\phi\sqrt{k}}\vc{D}(k) < \infty$ for some
$\phi > 0$, then (\ref{asymp-overline{p}(k)-01}) holds.
\end{prop}

Note that if $H \in \calL^2$, then $H_{\rme} \in \calL^2$ (see
Lemma~A.2 in \citealt{Masu09}). Note also that $H_{\rme} \in \calL^2$
if and only if $H_{\rme} \in \calL^{\mu}$ for some $\mu \ge 2$ (see
Proposition~\ref{rem-higher-order-L}~(i)). Thus conditions (i) and
(iii) of Theorem~\ref{thm-BMAP-GI-1-01} are weaker than condition (i)
of Proposition~\ref{Masu09-prop-01}. Further if $Q(x) = \phi\sqrt{x}$,
then condition (ii) of Theorem~\ref{thm-BMAP-GI-1-01} is reduced to
condition (ii) of Proposition~\ref{Masu09-prop-01}. As a result,
Theorem~\ref{thm-BMAP-GI-1-01} is a more general result than
Proposition~\ref{Masu09-prop-01}.

Actually, \citet{AsmuKlupSigm99} consider an M/GI/1 queue with arrival
rate $\lambda$ and service time distribution $H$, and the authors
prove that if $H_{\rme} \in \calL^2 \cap \calS$,
\[
\PP(L > k)
\simhm{k} {\rho \over 1 - \rho} \overline{H}_{\rme}(k/\lambda).
\]
Theorem~\ref{thm-BMAP-GI-1-01} includes this result as a special case
whereas Proposition~\ref{Masu09-prop-01} does not.

\subsubsection{Consistently varying service time}

\begin{thm}\label{thm-BMAP-GI-1-02}
Suppose that (i) $H_{\rme} \in \calC$ and $\int_0^{\infty}
\overline{H}_{\rme}(x)\rmd x < \infty$ and (ii) $\overline{\vc{D}}(k)
= o(\overline{H}_{\rme}(k))$.  We then have
(\ref{asymp-overline{A}_{e}(k)}). Further if (iii) there exists some
finite $\widetilde{\vc{d}}_H \ge \vc{0}$ such that
$\ooverline{\vc{D}}(k)\vc{e} \simhm{k} \overline{H}_{\rme}(k/\lambda)
\widetilde{\vc{d}}_H$, then
\begin{equation}
\ooverline{\vc{P}}(k)\vc{e} 
\simhm{k} 
\left(
\rho \vc{e} + h\widehat{\vc{P}}_{\rme}(1) \widetilde{\vc{d}}_H
\right) 
\overline{H}_{\rme}(k/\lambda),
\label{asymp-ooverline{A}(k)-02}
\end{equation}
and
\begin{equation}
\PP(L > k, J = i)
\simhm{k} 
{\rho + h\vc{\varpi} \widetilde{\vc{d}}_H \over 1 - \rho}
\varpi_i
\cdot \overline{H}_{\rme}(k/\lambda).
\label{asymp-overline{p}(k)-02}
\end{equation}
\end{thm}

\proof As in the proof of Theorem~\ref{thm-BMAP-GI-1-01}, let $T$
denote a nonnegative random variable distributed with $H_{\rme}$
independently of BMAP $\{\vc{C},\vc{D}(1),\vc{D}(2),\dots\}$. It is
easy to see that the conditions of Proposition~\ref{appendix-prop-02}
are satisfied. Using Proposition~\ref{appendix-prop-02}, we can obtain
(\ref{asymp-N(T)-01}) and thus (\ref{asymp-overline{A}_{e}(k)}) in the
same way as the proof of Theorem~\ref{thm-BMAP-GI-1-01}, where we do
not require condition (iii).

In addition, applying
Proposition~\ref{Masu11-prop} to (\ref{eqn-ooverline{A}(k)e}) and
using (\ref{asymp-overline{A}_{e}(k)}) and condition~(iii), we obtain
\begin{eqnarray*}
\ooverline{\vc{P}}(k)\vc{e} 
&\simhm{k}& 
h \left(
\vc{e}\vc{\varpi} \sum_{k=0}^{\infty}\overline{\vc{D}}(k)\vc{e}
+ \widehat{\vc{P}}_{\rme}(1) \widetilde{\vc{d}}_H
\right) 
\overline{H}_{\rme}(k/\lambda)
\nonumber
\\
&=& \left(
\rho \vc{e} + h\widehat{\vc{P}}_{\rme}(1) \widetilde{\vc{d}}_H
\right) 
\overline{H}_{\rme}(k/\lambda),
\end{eqnarray*}
where the last equality follows from (\ref{defn-lambda}) and $\rho =
\lambda h$. Therefore we have
(\ref{asymp-ooverline{A}(k)-02}). Combining
(\ref{asymp-ooverline{A}(k)-02}), (\ref{add-eqn-03a}) and
Corollary~\ref{coro-BMAP} yields
\[
\overline{\vc{y}}(k) 
\simhm{k} 
{\rho + h\vc{\varpi} \widetilde{\vc{d}}_H \over 1 - \rho}
\vc{\varpi}
\cdot \overline{H}_{\rme}(k/\lambda),
\]
which leads to (\ref{asymp-overline{p}(k)-02}).
\qed

\medskip

Suppose $\widetilde{\vc{d}}_H = \vc{0}$. It then follows that
asymptotic formula (\ref{asymp-overline{p}(k)-02}) in
Theorem~\ref{thm-BMAP-GI-1-02} has the same expression as
(\ref{asymp-overline{p}(k)-01}) in Theorem~\ref{thm-BMAP-GI-1-01}.
The two theorems assume that $\overline{\vc{D}}(k) =
o(\overline{H}_{\rme}(k))$ (see Remark~\ref{rem-thm-BMAP-GI-1-01} and
condition~(ii) of Theorem~\ref{thm-BMAP-GI-1-02}) and thus that the
service time distribution has a dominant impact on the tail of the
stationary queue length distribution.

Conversely, the following theorem assumes, as with
Theorem~\ref{rem-thm-BMAP-GI-1-01}, that the batch size distribution
has a dominant impact on the tail of the stationary queue length
distribution.
\begin{thm}\label{thm-BMAP-GI-1-03}
Suppose that conditions (i) and (ii) of Theorem~\ref{thm-BMAP-GI-1-02}
are satisfied. Further suppose that Assumption~\ref{assumpt-G} holds
for $G_{\rm de} \in \calS$ such that $\overline{H}_{\rme}(k/\lambda) =
o(\PP(G_{\rm de} > k))$.  We then have (\ref{asymp-ooverline{A}(k)-03})
and thus (\ref{asymp-overline{p}(k)-00}).
\end{thm}

\proof As shown in the proof of Theorem~\ref{thm-BMAP-GI-1-02}, the
asymptotics (\ref{asymp-overline{A}_{e}(k)}) holds under conditions
(i) and (ii) of Theorem~\ref{thm-BMAP-GI-1-02}. From
(\ref{asymp-overline{A}_{e}(k)}) and $\overline{H}_{\rme}(k/\lambda) =
o(\PP(G_{\rm de} > k))$, we have (\ref{add-eqn-02}), i.e.,
$\overline{\vc{P}}_{\rme}(k) = o(\PP(G_{\rm de} > k))$.  The rest of the
proof is the same as that of Theorem~\ref{thm-BMAP-GI-1-00}.  \qed

\section{Application to MAP/GI${}^{(a,b)}$/1 Queue}\label{sec-app-02}

In this section, we apply out main result to a single-server queue
with Markovian arrivals and the $(a,b)$-bulk-service rule, which is
denoted by MAP/GI${}^{(a,b)}$/1 queue (\citealt{Sing13}).

\subsection{Model description}

We assume that the arrival process is a Markovian arrival process
(MAP), which is a special case of the BMAP $\{\vc{C},\vc{D}(1),
\vc{D}(2),\dots\}$ (introduced in Section~\ref{sec-app-01}) such that
$\vc{D}(k) = \vc{O}$ for all $k\ge 2$. For convenience, we use the
symbols defined for the BMAP in Section~\ref{sec-app-01}, though we
denote, for simplicity, $\vc{D}(1)$ by $\vc{D}$. Thus the MAP is
characterized by $\{\vc{C},\vc{D}\}$. As with
Section~\ref{sec-app-01}, we assume that $\vc{C} + \vc{D}$ is
irreducible and that the arrival rate $\lambda$ is strictly positive,
i.e.,
\begin{equation}
\lambda = \vc{\varpi}\vc{D}\vc{e} > 0,
\label{defn-lambda-MAP}
\end{equation}
where $\vc{\varpi}$ is the unique stationary probability vector of
$\vc{C} + \vc{D}$.

We also assume that the server works according to the
$(a,b)$-bulk-service rule (\citealt{Sing13}). To explain the
$(a,b)$-bulk-service rule, we suppose that $l$ customers are waiting
in the queue at the completion of a service. The $(a,b)$-bulk-service
rule is as follows:
\begin{enumerate}
\item If $0 \le l < a$, the server keeps idle until the queue length
  is equal to the lower threshold $a$ and then starts serving all the
  $a$ customers when the queue length reaches $a$; and
\item If $l \ge a$, the server immediately starts serving $\min(l,b)$
  customers in the queue and makes the other $l-b$ customers (if any)
  be in the queue.
\end{enumerate}

The service times are assumed to be independent of the number of
customers in service and i.i.d.\ according to distribution function
$H$ with mean $h \in (0,\infty)$ and $H(0) = 0$. We assume that the
offered load $\rho = \lambda h$ satisfies
\begin{equation}
\rho < b,
\label{stability-cond-bulk}
\end{equation}
under which the system is stable (\citealt{Loyn62}).

It should be noted that since $\vc{D}(k) = \vc{O}$ for all $k\ge 2$, (\ref{defn-hat{P}(z)}) and (\ref{defn-hat{A}_e(z)}) are reduced to
\begin{eqnarray}
\widehat{\vc{P}}(z)
&=& \int_0^{\infty} \exp\{(\vc{C} + z\vc{D})x\} \rmd H(x),
\label{defn-hat{P}(z)-MAP}
\\
\widehat{\vc{P}}_{\rme}(z)
&=& \int_0^{\infty} \exp\{(\vc{C} + z\vc{D})x\} \rmd H_{\rme}(x).
\label{defn-hat{A}_e(z)-MAP}
\end{eqnarray}
In addition, since $\overline{\vc{D}}(0)=\vc{D}$ and $\overline{\vc{D}}(k)=\vc{O}$ for all $k \in \bbN$, it follows from Lemma~\ref{lem-P_e} that
\begin{equation}
\overline{\vc{P}}(k)\vc{e} 
= h\cdot\vc{P}_{\rme}(k)\vc{D}\vc{e},
\qquad k \in \bbZ_+.
\label{eqn-ooverline{A}(k)e-MAP}
\end{equation}

\subsection{Queue length process}

Let $L^{(a,b)}(t)$ ($t \ge 0$) denote the total number of customers
in the system at time $t$. Let $J(t)$ ($t \ge 0$) denote the state of the background Markov chain at time $t$. Let $0 = t_0 \le t_1 \le t_2 \le \cdots$ denote time points at each of which a service is completed.

Let $L_n^{(a,b)}$ and $J_n$ ($n \in \bbZ_+$) denote
\[
L_n^{(a,b)} = \lim_{\varepsilon \downarrow 0}L^{(a,b)}(t_n+\varepsilon),
\quad
J_n = \lim_{\varepsilon \downarrow 0}J(t_n+\varepsilon).
\]
Thus $L_n^{(a,b)}$ and $J_n$ denote the number of customers
in the queue and the state of the background Markov chain, respectively, immediately after the completion of the $n$th service. It follows (\citealt{Sing13}) that $\{(L_n^{(a,b)}, J_n);n\in\bbN_+\}$ is a discrete-time Markov chain with state space $\bbZ_+ \times \bbM$, whose transition probability matrix $\vc{T}_+^{(a,b)}$ is given by
\begin{equation}
\vc{T}_+^{(a,b)}
=
\left(
\begin{array}{@{\,\,}c@{\,\,}c@{\,\,}c@{\,\,}c@{\,\,}c@{\,\,}c@{\,\,}c@{\,\,}c@{\,\,}}
  \vc{P}_0(0) 
& \vc{P}_0(1) 
& \vc{P}_0(2) 
& \cdots
& \vc{P}_0(a) 
& \cdots
& \vc{P}_0(b) 
& \cdots
\\
  \vc{P}_1(0) 
& \vc{P}_1(1) 
& \vc{P}_1(2) 
& \cdots
& \vc{P}_1(a) 
& \cdots
& \vc{P}_1(b) 
& \cdots
\\
  \vdots
& \vdots
& \vdots
& \ddots
& \vdots
& \ddots
& \vdots
& \ddots
\\
  \vc{P}_{a-1}(0) 
& \vc{P}_{a-1}(1) 
& \vc{P}_{a-1}(2) 
& \cdots
& \vc{P}_{a-1}(a) 
& \cdots
& \vc{P}_{a-1}(b) 
& \cdots
\\
  \vc{P}(0) 
& \vc{P}(1) 
& \vc{P}(2) 
& \cdots
& \vc{P}(a) 
& \cdots
& \vc{P}(b) 
& \cdots
\\
  \vc{P}(0) 
& \vc{P}(1) 
& \vc{P}(2) 
& \cdots
& \vc{P}(a) 
& \cdots
& \vc{P}(b) 
& \cdots
\\
  \vdots
& \vdots
& \vdots
& \ddots
& \vdots
& \ddots
& \vdots
& \ddots
\\
  \vc{P}(0) 
& \vc{P}(1) 
& \vc{P}(2) 
& \cdots
& \vc{P}(a) 
& \cdots
& \vc{P}(b) 
& \cdots
\\
  \vc{O}	
&  \vc{P}(0) 
& \vc{P}(1) 
& \cdots
& \vc{P}(a-1) 
& \cdots
& \vc{P}(b-1) 
& \cdots
\\
   \vc{O}
&  \vc{O}	
&  \vc{P}(0) 
& \cdots
& \vc{P}(a-2) 
& \cdots
& \vc{P}(b-2) 
& \cdots
\\
  \vdots
& \vdots
& \vdots
& \ddots
& \vdots
& \ddots
& \vdots
& \ddots
\end{array}
\right),
\label{defn-T_bulk}
\end{equation}
where
\begin{equation}
\vc{P}_l(k) = \left[ (-\vc{C})^{-1}\vc{D} \right]^{a-l}\vc{P}(k),
\qquad l=0,1,\dots,a-1,~k\in\bbZ_+.
\label{defn-P_l(k)}
\end{equation}
Under the stability condition (\ref{stability-cond-bulk}), the Markov
chain $\{(L_n^{(a,b)}, J_n);n\in\bbN_+\}$ and thus $\vc{T}_+^{(a,b)}$
have the unique stationary distribution.  Let $\vc{y}(k)$ denote a $1
\times M$ vector such that
\[
[\vc{y}_+^{(a,b)}(k)]_i = \lim_{n\to\infty} 
\PP(L_n^{(a,b)} = k, J_n=i),
\qquad (k,i) \in \bbZ_+ \times \bbM.
\]

It should be noted that the stochastic process
$\{(L^{(a,b)}(t),J(t));t\ge 0\}$ is a semi-regenerative process with
the embedded Markov renewal process $\{(L_n^{(a,b)}, J_n,
t_n);n\in\bbZ_+\}$ (\citealt[Chapter 10]{Cinl75}). Note also that
$\{(L_n^{(a,b)}, J_n, t_n);n\in\bbZ_+\}$ is aperiodic because the
arrival process is Markovian (\citealt[Chapter 10, Definition
  2.22]{Cinl75}). Further the mean regenerative cycle (mean
inter-departure time) is given by
\begin{eqnarray}
\eta
:=
\lefteqn{
\sum_{k=0}^{\infty}
\sum_{i \in \bbM}[\vc{y}_+^{(a,b)}(k)]_{i} \cdot 
\EE[t_1 \mid L_0^{(a,b)}=k, J_0 = i]
}
\quad &&
\nonumber
\\
&=& h + \sum_{k=0}^{a-1}
\vc{y}_+^{(a,b)}(k) \cdot (-1)
\lim_{s\downarrow0}{\rmd \over \rmd s}
\left[ \int_0^{\infty} \rme^{-sx} \exp\{\vc{C} x\} \rmd x \vc{D}\right]^{a-k}
\vc{e}
\nonumber
\\
&=& h - \sum_{k=0}^{a-1}
\vc{y}_+^{(a,b)}(k) 
\lim_{s\downarrow0}{\rmd \over \rmd s}
\left[ (s\vc{I} - \vc{C} )^{-1}\vc{D}\right]^{a-k}
\vc{e}
\nonumber
\\
&=& h + \sum_{k=0}^{a-1}
\vc{y}_+^{(a,b)}(k) 
\sum_{l=0}^{a-k-1}
\left[ (- \vc{C} )^{-1}\vc{D}\right]^l (- \vc{C} )^{-2}\vc{D}
\vc{e}
\nonumber
\\
&=& h + \sum_{k=0}^{a-1}
\vc{y}_+^{(a,b)}(k) 
\sum_{l=0}^{a-k-1}
\left[ (- \vc{C} )^{-1}\vc{D}\right]^l (- \vc{C} )^{-1}
\vc{e}.
\label{defn-eta}
\end{eqnarray}
According to Theorem 6.12 in Chapter 10 of \citet{Cinl75}, we have for
$(k,j) \in \bbZ_+ \times \bbM$,
\begin{eqnarray}
\lefteqn{
[\vc{y}^{(a,b)}(k)]_j
}
\quad &&
\nonumber
\\
&:=& \lim_{t\to\infty} \PP(L^{(a,b)}(t) = k, J(t) = j)
\nonumber
\\
&=& {1 \over \eta}
\sum_{k=0}^{\infty}
\sum_{i \in \bbM}[\vc{y}_+^{(a,b)}(l)]_{i} \cdot 
\int_0^{\infty}\PP_{(l,i)}( L^{(a,b)}(x) = k, J(x) = j, t_1 > x) \rmd x,
\qquad
\label{defn-y^{(a,b)}(k)}
\end{eqnarray}
where $\PP_{(l,i)}(\,\cdot \,) = \PP(\,\cdot \mid L^{(a,b)}(0) = l,
J(0) = i)$.

We now define $\vc{P}(t,k)$ ($t \ge 0$, $k \in \bbZ_+$) as an $M
\times M$ matrix such that
\[
[\vc{P}(t,k)]_{i,j} = \PP(N(t) = k, J(t) = j \mid J(0) = i),
\qquad i,j \in \bbM.
\]
It then follows from (\ref{defn-y^{(a,b)}(k)}) that
\begin{eqnarray}
\vc{y}^{(a,b)}(k)
&=& {1 \over \eta}
\sum_{l=0}^k \vc{y}_+^{(a,b)}(l) \left[ (-\vc{C})^{-1}\vc{D} \right]^{k-l}
(-\vc{C})^{-1}, \qquad  0  \le k \le a-1, \quad
\label{eqn-y^{(a,b)}(k)-01}
\\
\vc{y}^{(a,b)}(a)
&=& {1 \over \eta}
\sum_{l=0}^a \vc{y}_+^{(a,b)}(l) \left[ (-\vc{C})^{-1}\vc{D} \right]^{a-l}
\int_0^{\infty} \vc{P}(x,0) \overline{H}(x) \rmd x,
\label{eqn-y^{(a,b)}(k)-02}
\\
\vc{y}^{(a,b)}(k)
&=& {1 \over \eta}
\sum_{l=0}^a \vc{y}_+^{(a,b)}(l) \left[ (-\vc{C})^{-1}\vc{D} \right]^{a-l}
\int_0^{\infty} \vc{P}(x,k-a) \overline{H}(x) \rmd x 
\nonumber
\\
&& \quad {} + {1 \over \eta}
\sum_{l=a+1}^k \vc{y}_+^{(a,b)}(l)
\int_0^{\infty} \vc{P}(x,k-l) \overline{H}(x) \rmd x, \quad k \ge a+1.
\qquad~~
\label{eqn-y^{(a,b)}(k)-03}
\end{eqnarray}
Note here that $H_{\rme}'(x) = h^{-1} \overline{H}(x)$ for $x \ge
0$. Note also that
\[
\int_0^{\infty} \vc{P}(x,k) H_{\rme}'(x) \rmd x
= \int_0^{\infty} \vc{P}(x,k) \rmd H_{\rme}(x)
= \vc{P}_{\rme}(k),
\]
where the last equality is due to (\ref{defn-hat{A}_e(z)}).  Thus
(\ref{eqn-y^{(a,b)}(k)-02}) and (\ref{eqn-y^{(a,b)}(k)-03}) can be
rewritten as
\begin{align}
\vc{y}^{(a,b)}(a)
&= {h \over \eta}
\sum_{l=0}^a \vc{y}_+^{(a,b)}(l) \left[ (-\vc{C})^{-1}\vc{D} \right]^{a-l}
\vc{P}_{\rme}(0),
&  &
\label{eqn-y^{(a,b)}(k)-02'}
\\
\vc{y}^{(a,b)}(k)
&= {h \over \eta}
\sum_{l=0}^a \vc{y}_+^{(a,b)}(l) \left[ (-\vc{C})^{-1}\vc{D} \right]^{a-l}
\vc{P}_{\rme}(k-a) &&
\nonumber
\\
& \quad {} + {h \over \eta}
\sum_{l=a+1}^k \vc{y}_+^{(a,b)}(l)\vc{P}_{\rme}(k-l), & k &\ge a+1.
\label{eqn-y^{(a,b)}(k)-03'}
\end{align}

\subsection{Asymptotic formulas for the queue length}

Let $\vc{S}(0)$ denote a $bM \times bM$ matrix such that
\begin{eqnarray}
\vc{S}(0) &=& 
\left(
\begin{array}{@{\,\,}c@{\,\,}c@{\,\,}c@{\,\,}c@{\,\,}c@{\,\,}c@{\,\,}c@{\,\,}c@{\,\,}}
  \vc{P}_0(0) 
& \vc{P}_0(1) 
& \vc{P}_0(2) 
& \cdots
& \vc{P}_0(a) 
& \cdots
& \vc{P}_0(b-1) 
\\
  \vc{P}_1(0) 
& \vc{P}_1(1) 
& \vc{P}_1(2) 
& \cdots
& \vc{P}_1(a) 
& \cdots
& \vc{P}_1(b-1) 
\\
  \vdots
& \vdots
& \vdots
& \ddots
& \vdots
& \ddots
& \vdots
\\
  \vc{P}_{a-1}(0) 
& \vc{P}_{a-1}(1) 
& \vc{P}_{a-1}(2) 
& \cdots
& \vc{P}_{a-1}(a) 
& \cdots
& \vc{P}_{a-1}(b-1) 
\\
  \vc{P}(0) 
& \vc{P}(1) 
& \vc{P}(2) 
& \cdots
& \vc{P}(a) 
& \cdots
& \vc{P}(b-1) 
\\
  \vc{P}(0) 
& \vc{P}(1) 
& \vc{P}(2) 
& \cdots
& \vc{P}(a) 
& \cdots
& \vc{P}(b-1) 
\\
  \vdots
& \vdots
& \vdots
& \ddots
& \vdots
& \ddots
& \vdots
\\
  \vc{P}(0) 
& \vc{P}(1) 
& \vc{P}(2) 
& \cdots
& \vc{P}(a) 
& \cdots
& \vc{P}(b-1) 
\end{array}
\right), \qquad~~
\label{defn-S(0)}
\end{eqnarray}
and let $\vc{S}(k)$ ($k \in \bbN$) denote a $bM \times M$ matrix such
that
\begin{eqnarray}
\vc{S}(k) &=& 
\left(
\begin{array}{c}
  \vc{P}_0(k+b-1)
\\
\vc{P}_1(k+b-1)
\\
\vdots
\\
\vc{P}_{a-1}(k+b-1)
\\
\vc{P}(k+b-1)
\\
\vc{P}(k+b-1)
\\
\vdots
\\
\vc{P}(k+b-1)
\end{array}
\right).
\label{defn-S(k)}
\end{eqnarray}
Further let $\vc{S}(-k)$ ($k=1,2,\dots,b$) denote an $M \times bM$
matrix such that
\begin{eqnarray}
\vc{S}(-k) &=& 
\left(
 \overbrace{\vc{O},\vc{O}, \dots, \vc{O}}^{k-1}, {\vc{P}(0)}, \vc{P}(1), \dots, \vc{P}(b-k)
\right).
\label{defn-S(-k)}
\end{eqnarray}
We then rewrite (\ref{defn-T_bulk}) as
\begin{equation}
\vc{T}_+^{(a,b)}
=
\left(
\begin{array}{c|cccc}
\vc{S}(0) & \vc{S}(1) & \vc{S}(2) & \vc{S}(3) & \cdots
\\
\hline
\vc{S}(-1) & \vc{P}(b) & \vc{P}(b+1) & \vc{P}(b+2) & \cdots
\\
\vc{S}(-2) & \vc{P}(b-1) & \vc{P}(b) & \vc{P}(b+1) & \cdots
\\
\vdots    & \vdots    &  \vdots   & \vdots   & \ddots
\\
\vc{S}(-b) & \vc{P}(1) & \vc{P}(2) & \vc{P}(3) & \cdots
\\
\vc{O} & \vc{P}(0) & \vc{P}(1) & \vc{P}(2) & \cdots
\\
\vc{O} &  \vc{O} & \vc{P}(0) & \vc{P}(1) &  \cdots
\\
\vdots    & \vdots    &  \vdots   & \vdots   & \ddots
\end{array}
\right),
\label{defn-T_bulk-02}
\end{equation}
which is a GI/G/1-type Markov chain without disasters.

\begin{lem}
Suppose that the arrival process is the MAP $\{\vc{C},\vc{D}\}$, i.e.,
a BMAP characterized by $\{\vc{C},\vc{D}(1), \vc{D}(2),\dots\}$ such
that $\vc{D}(k) = \vc{O}$ for all $k\ge 2$. If $H_{\rme} \in \calL^2$,
then
\begin{eqnarray}
\overline{\vc{P}}_{\rme}(k)
&\simhm{k}& \vc{e}\vc{\varpi} \cdot \overline{H}_{\rme}(k/\lambda),
\label{asymp-overline{P}_e(k)}
\\
\ooverline{\vc{P}}(k)
&\simhm{k}& \rho \vc{e} \cdot \overline{H}_{\rme}(k/\lambda),
\label{asymp-overline{P}(k)e}
\\
\ooverline{\vc{S}}(k)
&\simhm{k}& \rho \vc{e} \cdot \overline{H}_{\rme}(k/\lambda).
\label{asymp-overline{S}(k)e}
\end{eqnarray}
\end{lem}

\begin{proof}
Since conditions (i) and (ii) of Theorem~\ref{thm-BMAP-GI-1-01} are
satisfied, the asymptotic equation (\ref{asymp-overline{P}_e(k)})
hold.  Substituting (\ref{asymp-overline{P}_e(k)}) into
(\ref{eqn-ooverline{A}(k)e-MAP}) and using (\ref{defn-lambda-MAP}) and
$\rho=\lambda h$ yield
\[
\ooverline{\vc{P}}(k)\vc{e} 
\simhm{k} h \vc{e} \vc{\varpi}\vc{D}\vc{e}\cdot \overline{H}_{\rme}(k/\lambda)
= \rho \vc{e} \cdot \overline{H}_{\rme}(k/\lambda),
\]
which shows that (\ref{asymp-overline{P}(k)e}) holds. Further applying
(\ref{asymp-overline{P}(k)e}) to (\ref{defn-P_l(k)}) and using
$(-\vc{C})^{-1}\vc{D}\vc{e}=\vc{e}$, we obtain for $l=0,1,\dots,b-1$,
\[
\ooverline{\vc{P}}_l(k)\vc{e}
\simhm{k} \left[ (-\vc{C})^{-1}\vc{D} \right]^{a-l} \rho \vc{e} 
\cdot \overline{H}_{\rme}(k/\lambda)
= \rho \vc{e} \cdot \overline{H}_{\rme}(k/\lambda).
\]
Finally, incorporating this and (\ref{asymp-overline{P}(k)e}) into
(\ref{defn-S(k)}) yields (\ref{asymp-overline{S}(k)e}).
\end{proof}

\begin{thm}
If $H_{\rme} \in \calL^2 \cap \calS$, then
\begin{eqnarray}
\overline{\vc{y}}{}_+^{(a,b)}(k) 
&\simhm{k}&
{\rho \over b - \rho}\vc{\varpi}\cdot \overline{H}_{\rme}(k/\lambda),
\label{asymp-overline{y}_+^{(a,b)}(k)}
\\
\overline{\vc{y}}^{(a,b)}(k) 
&\simhm{k}& {h \over \eta}
{b \over b - \rho}\vc{\varpi}\cdot \overline{H}_{\rme}(k/\lambda).
\label{asymp-overline{y}^{(a,b)}(k)}
\end{eqnarray}
\end{thm}

\begin{proof}
Note that $\vc{T}_+^{(a,b)}$ in (\ref{defn-T_bulk-02}) is
equivalent to $\vc{T}$ in (\ref{defn-T}) with 
\begin{eqnarray}
\vc{A}(k) = 
\left\{
\begin{array}{ll}
\vc{P}(k+b), & k \ge -b,
\\
\vc{O}, & k \le -b-1,
\end{array}
\right.
\qquad
\vc{B}(k) = 
\left\{
\begin{array}{ll}
\vc{S}(k), & k \ge -b,
\\
\vc{O}, & k \le -b-1.
\end{array}
\right.
\label{relation-AB-P-bulk}
\end{eqnarray}
It then follows from (\ref{eqn-rho}) and (\ref{stability-cond-bulk})
that
\[
\vc{\varpi}\sum_{k\in\bbZ} k \vc{A}(k)\vc{e} 
= \vc{\varpi}\sum_{k=-b}^{\infty} k \vc{P}(k+b)\vc{e} 
= \rho - b < 0.
\]
It also follows from (\ref{asymp-overline{P}(k)e}),
(\ref{asymp-overline{S}(k)e}) and (\ref{relation-AB-P-bulk}) that
\[
\ooverline{\vc{A}}(k)\vc{e} 
\simhm{k} \rho \vc{e} \cdot \overline{H}_{\rme}(k/\lambda),
\quad 
\ooverline{\vc{B}}(k)\vc{e} 
\simhm{k} \rho \vc{e} \cdot \overline{H}_{\rme}(k/\lambda),
\]
where the dimensions of $\ooverline{\vc{A}}(k)\vc{e}$ and
$\ooverline{\vc{B}}(k)\vc{e}$ are different each other. Combining
these results and Theorem~\ref{thm-original} yields
\[
\overline{\vc{y}}{}_+^{(a,b)}(k) 
\simhm{k} 
{\rho\sum_{k=0}^{\infty}\vc{y}_+^{(a,b)}(k)\vc{e} \over b - \rho}\vc{\varpi}\cdot \overline{H}_{\rme}(k/\lambda)
= {\rho \over b - \rho}\vc{\varpi}\cdot \overline{H}_{\rme}(k/\lambda),
\]
which shows that (\ref{asymp-overline{y}_+^{(a,b)}(k)}) holds.

Next we prove (\ref{asymp-overline{y}^{(a,b)}(k)}).  From
(\ref{eqn-y^{(a,b)}(k)-03'}), we have for $k \ge a$,
\begin{eqnarray*}
\overline{\vc{y}}{}^{(a,b)}(k)
&=& {h \over \eta}
\sum_{l=0}^a \vc{y}_+^{(a,b)}(l) \left[ (-\vc{C})^{-1}\vc{D} \right]^{a-l}
\overline{\vc{P}}_{\rme}(k-a)
\nonumber
\\
&& \quad {} + {h \over \eta} \overline{\vc{y}_+^{(a,b)} \ast \vc{P}_{\rme}}(k)
- {h \over \eta}
\sum_{l=0}^a \vc{y}_+^{(a,b)}(l)\overline{\vc{P}}_{\rme}(k-l).
\end{eqnarray*}
Applying (\ref{asymp-overline{P}_e(k)}),
(\ref{asymp-overline{y}_+^{(a,b)}(k)}) and
Proposition~\ref{Masu11-prop} to the above equation and using the
long-tailed property of $H_{\rme}$, we obtain
\begin{eqnarray*}
\lim_{k\to\infty}
{ \overline{\vc{y}}{}^{(a,b)}(k) \over \overline{H}_{\rme}(k/\lambda) }
&=& {h \over \eta}
\sum_{l=0}^a \vc{y}_+^{(a,b)}(l) \left[ (-\vc{C})^{-1}\vc{D} \right]^{a-l}
\vc{e}\vc{\varpi} - {h \over \eta}
\sum_{l=0}^a \vc{y}_+^{(a,b)}(l)\vc{e}\vc{\varpi}
\nonumber
\\
&& \quad {} + {h \over \eta} 
\left[
\sum_{l=0}^{\infty}\vc{y}_+^{(a,b)}(l)\vc{e}\vc{\varpi}
+ {\rho \over b - \rho}\vc{\varpi} \sum_{l=0}^{\infty} \vc{P}_{\rme}(l)
\right]
\nonumber
\\
&=& {h \over \eta}
\sum_{l=0}^a \vc{y}_+^{(a,b)}(l) 
\vc{e}\vc{\varpi} - {h \over \eta}
\sum_{l=0}^a \vc{y}_+^{(a,b)}(l)\vc{e}\vc{\varpi}
 + {h \over \eta} 
\left[
\vc{\varpi}
+ {\rho \over b - \rho}\vc{\varpi} 
\right]
\nonumber
\\
&=&  {h \over \eta} {b \over b - \rho}\vc{\varpi},
\end{eqnarray*}
where the second equality follows from 
\[
(-\vc{C})^{-1}\vc{D}\vc{e}=\vc{e}, \quad
\vc{\varpi} \sum_{l=0}^{\infty} \vc{P}_{\rme}(l) = \vc{\varpi}, \quad
\sum_{l=0}^{\infty} \vc{y}_+^{(a,b)}(l) \vc{e}= 1.
\]
The proof is completed.
\end{proof}

\begin{rem}
Suppose $a=b=1$. It then follows that the MAP/G${}^{(a,b)}$/1 queue is
reduced to the standard MAP/GI/1 queue, which is a special case of the
BMAP/GI/1 queue.  Further from (\ref{defn-eta}) and
(\ref{eqn-y^{(a,b)}(k)-01}), we have
\begin{eqnarray}
1 
&=& { h \over \eta }  
+ {\vc{y}_+^{(1,1)}(0) (- \vc{C} )^{-1}
\vc{e} \over \eta } 
= { h \over \eta } + \vc{y}^{(1,1)}(0)\vc{e}
= { h \over \eta } + 1 - \rho,
\label{add-eqn-05}
\end{eqnarray}
where the last equality holds because $\vc{y}^{(1,1)}(0)\vc{e} = 1 -
\rho$ (due to Little's law). The equation (\ref{add-eqn-05}) yields
$h/\eta = \rho$. Substituting this into
(\ref{asymp-overline{y}^{(a,b)}(k)}), we have
\[
\overline{\vc{y}}^{(1,1)}(k) 
\simhm{k}
{\rho \over 1 - \rho}\vc{\varpi}\cdot \overline{H}_{\rme}(k/\lambda),
\]
which is consistent with (\ref{asymp-overline{p}(k)-01}) in
Theorem~\ref{thm-BMAP-GI-1-01}.
\end{rem}

\appendix

\section{Proofs}

\subsection{Proof of Lemma~\ref{lem-AD-BD}}\label{proof-lem-AD=BD}

We prove (\ref{limit-AD}) only. The proof of (\ref{limit-BD}) is
omitted because it is similar to that of (\ref{limit-AD}).

According to Proposition~\ref{prop-lim-L(ntau+l)}, we fix $\varepsilon
>0$ arbitrarily and $m_{\ast}:= m_{\ast} (\varepsilon)$ such that for
all $m \ge m_{\ast}$ and $l = 0,1,\dots,\tau-1$,
\begin{eqnarray}
&&
\vc{e}(\tau\vc{\psi} - \varepsilon\vc{e}^{\rmt})
\le
\sum_{l=0}^{\tau-1}
\vc{L} (\lfloor m / \tau \rfloor \tau + l) 
\le \vc{e}(\tau\vc{\psi} + \varepsilon\vc{e}^{\rmt}).
\label{eq-ep-D}
\end{eqnarray}
Further since $\vc{L}(m) \leq \vc{e}\vc{e}^{\rmt}$ for all $m \in
\bbN$, it follows from (\ref{eqn-lim-A(k)e}) and $Y \in \calL$ that
\begin{eqnarray*}
\limsup_{k \to \infty} \sum_{m=1}^{m_{\ast}-1} 
{ \overline{\vc{A}}(k+m) \vc{L}(m) \over \PP(Y > k) }
&\leq&
\sum_{m=1}^{m_{\ast}-1} \limsup_{k \to \infty} 
{ \ooverline{\vc{A}}(k+m-1)\vc{e}\vc{e}^{\rmt} 
- \ooverline{\vc{A}}(k+m)\vc{e}\vc{e}^{\rmt}  \over \PP(Y > k) }
= \vc{O},
\label{eqn-lim-AD-1}
\end{eqnarray*}
and thus 
\begin{eqnarray}
\lim_{k\to \infty} \sum_{m=1}^{\infty} 
{ \overline{\vc{A}} (k+m) \vc{L}(m) 
\over \PP(Y > k) }
=
\lim_{k \to \infty} \sum_{m=m_{\ast}}^{\infty}
{ \overline{\vc{A}} (k+m) \vc{L}(m) 
\over \PP(Y > k) }.
\label{add-eqn-07}
\end{eqnarray}

To prove (\ref{limit-AD}) it suffices to show that
for any fixed $\varepsilon > 0$,
\begin{eqnarray}
\limsup_{k \to \infty}
\sum_{m=m_{\ast}}^{\infty}
{\overline{\vc{A}}(k+m) \vc{L}(m) \over \PP(Y > k) }
&\le& 
\vc{c}{}_A(\vc{\psi} + \varepsilon\vc{e}^{\rmt}/\tau),
\label{add-eqn-13a}
\\
\liminf_{k \to \infty}
\sum_{m=m_{\ast}}^{\infty}
{\overline{\vc{A}}(k+m) \vc{L}(m) \over \PP(Y > k) }
&\ge& 
\vc{c}{}_A(\vc{\psi} - \varepsilon\vc{e}^{\rmt}/\tau).
\label{add-eqn-13b}
\end{eqnarray}
Indeed, letting $\varepsilon
\downarrow 0$ in (\ref{add-eqn-13a}) and (\ref{add-eqn-13b}) we obtain
\[
\lim_{k \to \infty} \sum_{m=m_{\ast}}^{\infty}
{ \overline{\vc{A}} (k+m) \vc{L}(m) 
\over \PP(Y > k) } = \vc{c}{}_A\vc{\psi}
= {\vc{c}_A\vc{\pi}(\vc{I} - \vc{R} ) (\vc{I} - \vc{\Phi}(0)) 
\over -\sigma },
\]
where the second equality is due to (\ref{defn-psi}). Substituting the
obtained equation into (\ref{add-eqn-07}), we have (\ref{limit-AD}).

We first prove (\ref{add-eqn-13a}).  By definition,
$\{\overline{\vc{A}}(k);k\in\bbZ_+\}$ is nonincreasing. We thus obtain
\begin{eqnarray*}
\sum_{m=m_{\ast}}^{\infty}
\overline{\vc{A}} (k+m) \vc{L}(m) 
&\le&
\sum_{n=  \lfloor m_{\ast}/\tau \rfloor}^{\infty}
\sum_{l=0}^{\tau-1}
\overline{\vc{A}}(k + n\tau + l) \vc{L}(n\tau + l) 
\nonumber
\\
&\le& 
\sum_{n=\lfloor m_{\ast}/\tau \rfloor}^{\infty}
\overline{\vc{A}}(k + n\tau) 
\sum_{l=0}^{\tau-1}
\vc{L}(n\tau + l)
\nonumber
\\
&\le& 
\sum_{n=\lfloor m_{\ast}/\tau \rfloor}^{\infty}
{1 \over \tau}
\sum_{i=0}^{\tau-1} \overline{\vc{A}}(k + n\tau - i)  
\cdot
\sum_{l=0}^{\tau-1} \vc{L}(n\tau + l).
\end{eqnarray*}
Substituting (\ref{eq-ep-D}) into the above
inequality yields
\begin{eqnarray}
\sum_{m=m_{\ast}}^{\infty}
{ \overline{\vc{A}} (k+m) \vc{L}(m) 
\over \PP(Y > k) }
&\le& 
\sum_{n=\lfloor m_{\ast}/\tau \rfloor}^{\infty}
\sum_{i=0}^{\tau-1}
{ \overline{\vc{A}}(k + n\tau - i) \vc{e}
\over \PP(Y > k) }
(\vc{\psi} + \varepsilon \vc{e}^{\rmt}/\tau).
\nonumber
\\
&=& 
{ \ooverline{\vc{A}}(k + \lfloor m_{\ast}/\tau \rfloor \tau - \tau) \vc{e}
\over \PP(Y > k) }
(\vc{\psi} + \varepsilon \vc{e}^{\rmt}/\tau).
\label{add-eqn-23}
\end{eqnarray}
From (\ref{add-eqn-23}), (\ref{eqn-lim-A(k)e}) and $Y \in
\calL$, we have (\ref{add-eqn-13a}).

Next we prove (\ref{add-eqn-13b}). Since $\{\overline{\vc{A}}(k)\}$ is
nonincreasing, we have
\begin{eqnarray*}
\sum_{m=m_{\ast}}^{\infty}
\overline{\vc{A}} (k+m) \vc{L}(m) 
&\ge&
\sum_{n=\lceil m_{\ast}/\tau \rceil}^{\infty}
\sum_{l=0}^{\tau-1}
\overline{\vc{A}}(k + n\tau + l) \vc{L}(n\tau + l) 
\nonumber
\\
&\ge& 
\sum_{n=\lceil m_{\ast}/\tau \rceil}^{\infty}
\overline{\vc{A}}(k + n\tau + \tau + 1) 
\sum_{l=0}^{\tau-1} \vc{L}(n\tau + l)
\nonumber
\\
&\ge& 
\sum_{n=\lceil m_{\ast}/\tau \rceil}^{\infty}
{1 \over \tau}
\sum_{i=1}^{\tau} 
\overline{\vc{A}}(k + n\tau + \tau + i) 
\cdot \sum_{l=0}^{\tau-1} \vc{L}(n\tau + l).
\end{eqnarray*}
Combining this with (\ref{eq-ep-D}) yields
\begin{eqnarray*}
\sum_{m=m_{\ast}}^{\infty}
{ \overline{\vc{A}} (k+m) \vc{L}(m) 
\over \PP(Y > k) }
&\ge& 
\sum_{n=\lceil m_{\ast}/\tau \rceil+1}^{\infty}
\sum_{i=1}^{\tau}
{ \overline{\vc{A}}(k + n\tau + i) \vc{e}
\over \PP(Y > k) }
(\vc{\psi} - \varepsilon \vc{e}^{\rmt}/\tau).
\nonumber
\\
&=& 
{ \ooverline{\vc{A}}(k + \lceil m_{\ast}/\tau \rceil \tau + \tau) \vc{e}
\over \PP(Y > k) }
(\vc{\psi} - \varepsilon \vc{e}^{\rmt}/\tau).
\end{eqnarray*}
Therefore similarly to (\ref{add-eqn-13a}), we can obtain
(\ref{add-eqn-13b}).  \qed

\section{Cumulative process sampled at heavy-tailed random times}

This section summarizes some of the results presented in
\citet{Masu13}, which are used in Sections \ref{sec-app-01} and \ref{sec-app-02}.

Let $\{B(t);t\ge0\}$ denote a stochastic process on $
(-\infty,\infty)$, where $|B(0)| < \infty$ with probability one
(w.p.1). We assume that there exist regenerative points $0 \le \tau_0
< \tau_1 < \tau_2 < \cdots$ such that $\{B(t+\tau_n)-B(\tau_n); t \ge
0\}$ ($n \in \bbZ_+$) is independent of $\{B(u);0 \le u < \tau_n\}$
and is stochastically equivalent to $\{B(t+\tau_0)-B(\tau_0); t \ge
0\}$. The process $\{B(t);t\ge0\}$ is called {\it (regenerative)
  cumulative process}, which is introduced by \citet{Smit55}.

Let $\Delta\tau_0 = \tau_0$ and $\Delta\tau_n = \tau_n-\tau_{n-1}$ for
$n\in\bbN$. Let
\begin{eqnarray*}
\Delta B_n 
&=& \left\{
\begin{array}{ll}
B(\tau_0), & n = 0,
\\
B(\tau_n)-B(\tau_{n-1}), & n \in \bbN,
\end{array}
\right.
\quad
\Delta B_n^{\ast} 
=
\left\{
\begin{array}{ll}
\dm \sup_{0 \le t \le \tau_0} \max(B(t), 0), & n = 0,
\\
\dm\sup_{\tau_{n-1} \le t \le \tau_n} B(t) - B(\tau_{n-1}), & n \in \bbN.
\end{array}
\right.
\nonumber
\end{eqnarray*}
It is easy to see that $\Delta B_n^{\ast} \ge \Delta B_n$ for
$n\in\bbZ_+$ and that $\{\Delta \tau_n;n\in\bbN\}$ (resp.\ $\{\Delta
B_n;n\in\bbN\}$ and $\{\Delta B_n^{\ast};n\in\bbN\}$) is a sequence of
i.i.d.\ random variables, which is independent of $\Delta \tau_0$
(resp.\ $\Delta B_0$ and $\Delta B_0^{\ast}$).

\begin{rem}\label{rem-BMAP}
The counting process $\{N(t);t \ge 0\}$ of BMAP
$\{\vc{C},\vc{D}(1),\vc{D}(2),\dots\}$ is a cumulative process such
that regenerative points are hitting times to any fixed background
state and the regenerative cycle follows a phase-type distribution (see equations (3.3)--(3.5) in \citealt{Masu13}).
\end{rem}

We now assume that
\begin{eqnarray*}
&&
\PP(0\le \Delta \tau_n < \infty)   
= \PP(0 \le \Delta B_n^{\ast} < \infty) = 1~~(n=0,1),
\nonumber
\\
&&
\EE[|\Delta B_1|] <\infty,~~~
0 < \EE[\Delta \tau_1] < \infty,~~~
b := {\EE[\Delta B_1] \big/ \EE[\Delta \tau_1] } > 0.
\label{cond-SLLN-01}
\end{eqnarray*}
We then obtain the following results.
\begin{prop}[\citealt{Masu13}, Theorem~3.3]\label{appendix-prop-01}
Suppose that $T$ is a nonnegative random variable independent of
$\{B(t);t\ge0\}$. Further suppose that (i) $T \in \calL^{\mu}$ for
some $\mu\ge2$; (ii) $\EE[(\Delta \tau_1)^2] < \infty$ and
$\EE[(\Delta B_1)^2] < \infty$; and (iii) $\EE[\exp\{Q(\Delta
  B_n^{\ast})\}] < \infty$ ($n=0,1$) for some cumulative hazard
function $Q \in \SC$ such that $x^{1/\mu} = O(Q(x))$. We then have
$\PP(B(T) > bx) \simhm{x} \PP(T > x)$.
\end{prop}

\begin{coro}\label{appendix-coro-01}
Suppose that $T$ is a nonnegative random variable independent of
$\{(N(t),J(t));t\ge0\}$, where $\{N(t)\}$ and $\{J(t)\}$ denote the
counting process and the background Markov chain, respectively, of
BMAP $\{\vc{C},\vc{D}(1),\vc{D}(2),\dots\}$ introduced in
subsection~\ref{subsec-BMAP}. Suppose that (i) $T \in \calL^{\mu}$ for
some $\mu\ge2$; and (ii) $\sum_{k=1}^{\infty} \exp\{Q(k)\} \vc{D}(k) <
\infty$ ($n=0,1$) for some cumulative hazard function $Q \in \SC$ such
that $x^{1/\mu} = O(Q(x))$. We then have $\PP(N(T) > k) \simhm{k}
\PP(T > k/\lambda)$.
\end{coro}

\begin{proof}
It suffices to prove that conditions (i)--(iii) of Proposition~\ref{appendix-prop-01} are satisfied. For this purpose, fix $B(t) = N(t)$ for $t\ge 0$. Since the regenerative cycle follows a phase-type distribution (see Remark~\ref{rem-BMAP}), we have $\EE[(\Delta \tau_1)^2] < \infty$. Further since $\{B(t)=N(t);t\ge0\}$ is nondecreasing, we have $\Delta B_n^{\ast} = \Delta B_n$ for all $n \in \bbZ_+$. Therefore it follows from the renewal reward
theorem (see, e.g., \citealt[Chapter 2, Theorem 2]{Wolf89}) that
\[
{\EE[\Delta B_1^{\ast}] \over \EE[\Delta \tau_1]} = \lambda \in (0,\infty),
\qquad
{\EE[\exp\{Q(\Delta B_1^{\ast})\}] \over \EE[\Delta \tau_1]}
= \vc{\pi}\sum_{k=1}^{\infty} \exp\{Q(k)\} \vc{D}(k) \vc{e} 
< \infty,
\]
which lead to $\EE[\exp\{Q(\Delta B_1^{\ast})\}] < \infty$ and thus $\EE[(\Delta B_1)^2] < \infty$. 

It remains to prove $\EE[\exp\{Q(\Delta B_0^{\ast})\}] < \infty$. Let $i_0$ denote the background state at regenerative points, i.e., $J(\tau_n) = i_0$ for all $n \in \bbZ_+$. Suppose that there exists some $i \in \bbM$ such that
\begin{equation}
\EE[\exp\{Q(N(\tau_0))\} \cdot \dd{1}(J(\tau_0) = i_0) \mid J(0) = i] = \infty,
\label{add-eqn-04a}
\end{equation}
where $\tau_0=\inf\{t \ge 0; J(t) = i_0\}$. Let $T_i^{\geqslant\tau_0}
= \inf\{t \ge \tau_0; J(t) = i\}$. Since the background Markov chain
is irreducible, we have
\begin{equation}
\PP(T_i^{\geqslant\tau_0} < \tau_1 \mid J(\tau_0)=i_0) > 0,
\label{add-eqn-04b}
\end{equation}
where $\tau_1=\inf\{t \ge \tau_0; J(t) = i_0\}$.  
It follows from $\Delta B_1^{\ast} = N(\tau_1) - N(\tau_0)$,
 (\ref{add-eqn-04a}) and
(\ref{add-eqn-04b}) that
\begin{eqnarray*}
\EE[\exp\{Q(\Delta B_1^{\ast})\}]
&=&
\EE[\exp\{Q(N(\tau_1) - N(\tau_0) )\} ]
\nonumber
\\
&\ge& \PP(T_i^{\geqslant\tau_0} < \tau_1 \mid J(\tau_0)=i_0) 
\nonumber
\\
&& {} \times 
\EE[\exp\{Q( N(\tau_1) - N( T_i^{\geqslant\tau_0} ) )\} 
\mid J(T_i^{\geqslant\tau_0})=i, T_i^{\geqslant\tau_0} < \tau_1]
\nonumber
\\
&=& \PP(T_i^{\geqslant\tau_0} < \tau_1 \mid J(\tau_0)=i_0) 
\nonumber
\\
&& {} \times 
\EE[\exp\{Q( N(\tau_0) )\} \mid J(0))=i] = \infty,
\end{eqnarray*}
which is inconsistent with $\EE[\exp\{Q(\Delta B_1^{\ast})\}] <
\infty$. Thus (\ref{add-eqn-04a}) is not true. As a result, for any $i
\in \bbM$, we have $\EE[\exp\{Q(N(\tau_0))\} \cdot \dd{1}(J(\tau_0) =
  i_0) \mid J(0) = i] = \infty$, which implies that
$\EE[\exp\{Q(\Delta B_0^{\ast})\}] < \infty$.
\end{proof}

A similar result is presented in \citet{Masu13}.
\begin{prop}[\citealt{Masu13}, Corollary~3.1]\label{appendix-prop-02}
Suppose that $T$ is a nonnegative random variable independent of
$\{(N(t),J(t));t\ge0\}$, where $\{N(t)\}$ and $\{J(t)\}$ denote the
counting process and the background Markov chain, respectively, of
BMAP $\{\vc{C},\vc{D}(1),\vc{D}(2),\dots\}$ introduced in
subsection~\ref{subsec-BMAP}. Suppose that (i) $T \in \calC$; (ii)
$\EE[T] < \infty$; and (iii) $\overline{\vc{D}}(k)= o( \PP(T > k)
)$. We then have $\PP(N(T) > k) \simhm{k} \PP(T > k/\lambda)$.
\end{prop}


\section*{Acknowledgments}
Research of the author was supported in part by Grant-in-Aid for Young
Scientists (B) of Japan Society for the Promotion of Science under
Grant No.~24710165.

%
%
%

\end{document}